\documentclass[12pt,notitlepage]{amsart}
\usepackage{latexsym,amsfonts,amssymb,amsmath,amsthm}
\usepackage{color}
\usepackage[inner=1.0in,outer=1.0in,bottom=1.0in, top=1.0in]{geometry}

\usepackage[usenames,dvipsnames]{xcolor}

\newcommand{\mz}{\ensuremath{\mathbb Z}}
\newcommand{\mr}{\ensuremath{\mathbb R}}
\newcommand{\mh}{\ensuremath{\mathbb H}}

\newcommand{\shortmod}{\ensuremath{\negthickspace \negthickspace \negthickspace \pmod}}

\newcommand{\half}{\ensuremath{ \frac{1}{2}}}

\newcommand{\intR}{\int_{-\infty}^{\infty}}

\newcommand{\sumstar}{\sideset{}{^*}\sum}

\DeclareMathOperator{\vol}{vol}

\theoremstyle{plain}
	\newtheorem{mytheo}{Theorem} [section]
	
	\newtheorem{myprop}[mytheo]{Proposition}
	\newtheorem{mycoro}[mytheo]{Corollary}
     \newtheorem{mylemma}[mytheo]{Lemma}
	
	\newtheorem{myconj}[mytheo]{Conjecture}

\theoremstyle{remark}

\numberwithin{equation}{section}
\begin{document}

%
 \author{Matthew P. Young}
 \address{Department of Mathematics \\
 	  Texas A\&M University \\
 	  College Station \\
 	  TX 77843-3368 \\
 		U.S.A.}
 \email{myoung@math.tamu.edu}
 \thanks{This material is based upon work supported by the National Science Foundation under agreement No. DMS-1401008.  Any opinions, findings and conclusions or recommendations expressed in this material are those of the author and do not necessarily reflect the views of the National Science Foundation.
 }


 \begin{abstract}
We show an asymptotic formula for the $L^2$ norm of the Eisenstein series $E(z,1/2+iT)$ restricted to a segment of a geodesic connecting infinity and an arbitrary real.  For generic geodesics of this form, the asymptotic formula shows that the Eisenstein series satisfies restricted QUE.  On the other hand, for rational geodesics, the Eisenstein series does not satisfy restricted quantum ergodicity.

As an application, we show that the zero set of the Eisenstein series intersects every such geodesic segment, provided $T$ is large.  

We also make analogous conjectures for the Maass cusp forms.
In particular, we predict that cusp forms do not satisfy restricted quantum ergodicity for rational geodesics.
 \end{abstract}
 \title{Equidistribution of Eisenstein series on geodesic
segments}
 \maketitle
 
\section{Introduction}
 \subsection{Background and previous results}
The behavior of eigenfunctions of the Laplacian on an arithmetic surface $\Gamma \backslash \mh$ is a rich subject with connections to analytic number theory and spectral geometry.   One of the motivating questions in the area is the Quantum Unique Ergodicity (QUE) conjecture of Rudnick and Sarnak \cite{RudnickSarnak}.  The QUE conjecture is
now a theorem of Lindenstrauss \cite{Lindenstrauss} (complete in the compact setting), with input by Soundararajan \cite{SoundQUE} to treat non-cocompact congruence subgroups.  
Arithmetic cases of the QUE conjecture are known to be nearly equivalent to subconvexity for certain $L$-functions, via Watson's formula \cite{Watson}.  This is one well-known example of a nontrivial connection between quantum chaos and the analytic theory of $L$-functions.
We refer the reader to the survey articles \cite{SarnakSpectra} \cite{SarnakQUE} for more background and motivation.

The QUE conjecture concerns the equidistribution of eigenfunctions restricted to Jordan measurable sets.  A more advanced question is to study the behavior of restrictions of eigenfunctions to lower-dimensional submanifolds.
For instance, one may fix a curve on $\Gamma \backslash \mh$, and study the analytic behavior of the eigenfunctions restricted to the curve.  Such questions were initially raised by Reznikov \cite{Reznikov} who used representation theory techniques for their study.  Later, Ghosh, Reznikov, and Sarnak \cite{GRS} obtained estimates (both upper and lower bounds) for the $L^2$ norm of a Maass form along special curves such as horocycles or geodesic segments, with the application of counting sign changes of the Maass form along the curve (see also work of J. Jung \cite{JJ} for further progress in this direction).  Marshall \cite{Marshall} has used the amplification method to bound from above the $L^2$ norm of a Maass form restricted to an arbitrary geodesic segment, with a power saving over the ``local'' bound (which is the universal bound of \cite{BGT} holding for a general compact Riemannian manifold).  Toth and Zelditch \cite{TothZelditch} and Dyatlov and Zworski \cite{DZ}  have shown a quantum ergodic restriction theorem that in particular holds on a generic geodesic in a compact hyperbolic surface.

The author \cite{Young} recently showed that the Eisenstein series $E_T(z) = E(z, 1/2 + iT)$ for $\Gamma = SL_2(\mz)$ equidistributes along the geodesic connecting $0$ and $i \infty$ (this is an analog of QUE but along the curve).  The Eisenstein series is a good test case for more advanced questions since it behaves in many respects like the cusp forms, yet is more explicit and easier to handle.  In this paper, we study the behavior of the Eisenstein series restricted to a geodesic segment of the form $\{x + iy : 0 < \alpha < y < \beta \}$.


To get our bearings, we quote two results from the literature.  Suppose that $\phi: SL_2(\mz)\backslash \mh \rightarrow \mr$ is smooth and compactly-supported.  Luo and Sarnak \cite{LuoSarnakQUE} showed QUE for Eisenstein series on $SL_2(\mz) \backslash \mh$, that is,
\begin{equation}
\label{eq:ETQUE}
\int_{SL_2(\mz)\backslash \mh} \phi(z) |E_T(z)|^2 \frac{dx dy}{y^2} = \frac{3}{\pi} \log(1/4 + T^2)  \int_{SL_2(\mz)\backslash \mh} \phi(z)  \frac{dx dy}{y^2} + o(\log T),
\end{equation}
as $T \rightarrow \infty$.  The interested reader may find a full main term with a power saving error term for \eqref{eq:ETQUE} in \cite[(4.13)]{Young}.   We also mention that Jakobson \cite{Jakobson} proved QUE for Eisenstein series on the unit tangent bundle $SL_2(\mz) \backslash SL_2(\mr)$.
Suppose that $\psi$ is a smooth function with support on $[\alpha, \beta]$.
Define
\begin{equation}
\label{eq:IpsixTdef}
 I_{\psi}(x,T) = \int_0^{\infty} \psi^2(y)
|E_T(x+iy)|^2 \frac{dy}{y}.
\end{equation}
The author \cite{Young} showed
\begin{equation}
\label{eq:YoungQUEcitation}
I_{\psi}(0,T) = 2 \frac{3}{\pi} \log(1/4 + T^2) \int_0^{\infty} \psi^2(y) \frac{dy}{y} + o(\log T).
\end{equation}
Again, one has a fully developed main term with a power saving error term in \eqref{eq:YoungQUEcitation}. 
Comparing \eqref{eq:ETQUE} and \eqref{eq:YoungQUEcitation}, 
note the Eisenstein series (squared) is twice as large on the distinguished geodesic $x=0$ as it is on any fixed disk in $SL_2(\mathbb{Z}) \backslash \mh$.

\subsection{The main asymptotic result}
Our main result is an asymptotic for $I_{\psi}(x,T)$ for arbitrary $x$.  
The shape of the main term depends on rational approximations to $2x$.
  By the Dirichlet approximation theorem, for each real $Q \geq 1$, there exist integers $a,q$ so that
\begin{equation}
\label{eq:DirichletApproximation}
 2x = \frac{a}{q} + \theta, \quad \text{where} \quad (a,q) = 1, \quad 1 \leq q \leq Q, \quad |q \theta| \leq Q^{-1}. 
\end{equation}
\begin{mytheo}
\label{thm:QUE}
Let $x \in \mr$ and $T \geq 1$.
Choose $Q=Q(T)$ so that $T^{\delta} \ll Q \ll T^{1/3-\delta}$ (where $\delta > 0$ is fixed),  and let $a,q, \theta$ satisfy \eqref{eq:DirichletApproximation}.
Then, as $T \rightarrow \infty$, we have
\begin{equation}
\label{eq:QUE}
 I_{\psi}(x,T)  =   \frac{3}{\pi} \log(1/4 + T^2) \int_0^{\infty} \psi^2(y) \Big[1 + B_q(1,T) J_0\Big(\frac{\theta T}{y}\Big)\Big] \frac{dy}{y} 
  + O((\log T)^{\frac{35}{36} + \varepsilon}),
\end{equation}
where $B_q(s,T)$ is a finite Euler product defined by \eqref{eq:BqDef} below, and $J_0(z)$ is the Bessel function of the first kind.  The implied constant is uniform in $x$.
\end{mytheo}
We now discuss in detail the behavior of the right hand side of \eqref{eq:QUE}. 
 If $x=0$, then $\theta = 0$, $q=1$, $B_q(1,T) = 1$, and $J_0(0) = 1$, so this recovers the main term from \eqref{eq:YoungQUEcitation}.  More generally, if $x$ is rational and $T$ (hence $Q$) is sufficiently large, then $\theta = 0$, and we see a  family of generalizations of the main term appearing in \eqref{eq:YoungQUEcitation}.  However, one should observe that $B_q(1,T)$ fluctuates as a function of $T$; for instance, if $q=p$ is prime, then $B_q(1,T)$ simplifies as
\begin{equation}
\label{eq:Bp1T}
 B_p(1,T)  = \frac{1}{p+1} (1 + 2 \cos(T \log p) -p^{-1}).
 \end{equation}
This fluctuation shows that the Eisenstein series does not satisfy restricted quantum ergodicity (not to mention, QUE) on rational geodesics with $q > 1$.  Note $2 \cos(T \log p) $ is the eigenvalue of the Hecke operator $T_p$ applied to $E_T$.

We show in Section \ref{section:BqBound} below that $|B_q(1,T)| < 1$ for all $q \geq 2$, and it is well-known that $|J_0(x)| \leq 1$ for all real $x$, so the secondary main term with $B_q(1,T) J_0(y^{-1} \theta T)$ is never larger than the first main term.  
In fact, it is easy to show the bound $B_q(1, T) \ll q^{-1+\varepsilon}$.  

For almost all $x$, in a strong sense, the secondary term may be omitted in \eqref{eq:QUE}, as we now explain.
Note that if $q \gg (\log T)^{1/36}$, then this secondary term of magnitude $\ll q^{-1+\varepsilon} \log T$ is smaller than the error term.
As for the $J_0$-Bessel aspect, if $|\theta T| \gg (\log T)^{\varepsilon}$, then repeated integration by parts shows $\int \psi^2(y) J_0(\theta T/y) \frac{dy}{y} \ll (\log T)^{-100}$, and again we see that this secondary term is smaller than the error term.
This analysis shows that, for a given large $T \gg 1$, and each rational number $a/q$ with $q \ll (\log T)^{1/36}$, there is a small interval surrounding $a/q$ of radius $\approx T^{-1}$ so that the secondary main term is relevant.  Outside of these small intervals, the secondary term may be absorbed into the error term.  
Precisely,
for each large $T$ there exists a set $\mathcal{E} \subset [0,1]$ of measure $\ll \frac{(\log T)^{1/18+\varepsilon}}{T}$, so that if $x \in [0,1]\setminus \mathcal{E}$, then 
\begin{equation*}
I_{\psi}(x,T)  =   \frac{3}{\pi} \log(1/4 + T^2) \int_0^{\infty} \psi^2(y) \frac{dy}{y} 
  + O((\log T)^{\frac{35}{36} + \varepsilon}).
\end{equation*}
This is the ``generic" case where $x$ does not have a close approximation by a rational with small denominator.  Here $\mathcal{E}$ consists of intervals of radius $T^{-1} (\log T)^{\varepsilon}$ around each rational $a/q \in [0,1]$ with $q \ll (\log T)^{1/36}$.  

It may also be instructive to compare the special
 cases $x=0$ and $x=2^{-101}$ (say), in an informal manner.  Intuitively, it should be clear that for small values of $T$, then $I_{\psi}(0,T) \approx I_{\psi}(2^{-101}, T)$.  
 Since $E_T$ is an eigenfunction of the Laplacian with eigenvalue $1/4 + T^2$, it is expected  
   to remain roughly constant on the Planck scale (or de Broglie wavelength) of $1/T$ (see \cite[p.289]{HR}).  We therefore expect that once $T$ is roughly of size $2^{100}$ then $I_{\psi}(0,T)$ should be uncorrelated with $I_{\psi}(2^{-101}, T)$.  Now let us examine the statement of Theorem \ref{thm:QUE} with $x=2^{-101}$.  For small values of $T$, then $Q$ is also small, and this forces the rational approximation of $2x$ to be $a/q = 0/1$, whence $\theta = 2x = 2^{-100}$.  By the remarks on the behavior of the $J_0$-Bessel aspect from a previous paragraph, when $T$ is somewhat larger than $2^{100}$, then the secondary main term is small, as expected.
Once $Q > 2^{-100}$, the rational approximation $a/q$ must become $2^{-100}$, which invokes a change in the value of $B_q(1,T)$, as well as $\theta$.
Note that
it is only when $T$ is roughly $\exp(2^{100 \times 36})$ that the secondary term with $B_{2^{100}}(1,T)$ is potentially larger than the error term of \eqref{eq:QUE}.

It is also interesting to 
classify irrational $x$'s for which there are arbitrarily large values of $T$ so that $ q \leq (\log T)^{\delta}$, and $|\theta T| \leq 1$.  
These are, essentially, the properties on $x$ that require the presence of the secondary term with $B_q(1,T) J_0(y^{-1} \theta T)$ in \eqref{eq:QUE} for arbitrarily large values of $T$.
These conditions are equivalent to the existence of integers $1 \leq q_1 < q_2 < \dots$ so that
\begin{equation*}
 |\theta_n q_n| \leq \frac{q_n}{\exp(q_n^{1/\delta})},
\end{equation*}
where $2x = \frac{a_n}{q_n} + \theta_n$.  One can easily construct such $x$ by a continued fraction expansion $2x= [ b_0;b_1, b_2, \dots]$, where $b_{n+1} \geq \exp(q_n^{1/\delta})$.  



\subsection{Sign changes of the Eisenstein series}
In \cite{JungYoung}, a variant on \eqref{eq:YoungQUEcitation} was proven, where the test function $\psi$ was allowed to have support shrinking with $T$, and also the product of the Eisenstein series at two slightly shifted points was averaged.  As a consequence, it was shown that $E_T^*(iy)$ (here $E_T^*$ is real-valued and has the same absolute value as $E_T$) changes sign $\gg T^{\nu}$ times for $1 \leq y \leq 3$ for fixed $0 \leq \nu < 1/51$, and sufficiently large $T$.  It is a special property of the geodesic $x=0$ that sign changes along this curve produce nodal domains (see \cite{GRS} for details).  We are naturally led to consider a variant on Theorem \ref{thm:QUE} but with the product of two shifted Eisenstein series.  See Theorem \ref{thm:shiftedQUE} for the statement of this result, whose proof is a small modification of that of Theorem \ref{thm:QUE}.  As a consequence we have:
\begin{mycoro}
\label{coro:signchanges}
Fix $0 < \alpha < \beta$.  There exists a large $T_0$ depending on $\alpha,\beta$ only so that if $T \geq T_0$, then
$E_T^*(x+iy)$ changes sign at least once for $\alpha < y < \beta$.  
\end{mycoro}
Hence, the nodal set of $E_T^*$ intersects every geodesic segment of the form $x+iy: \alpha < y < \beta$, provided $T$ is sufficiently large.  We emphasize that $T_0$ is independent of $x$.  

In Corollary \ref{coro:signchanges}, compared to the result of \cite{JungYoung}, the gain in generality in $x$ has a tradeoff in that the number of sign changes is no longer a power of $T$.  Since the error term in Theorem \ref{thm:QUE} saves a small power of $\log T$, one would expect that, with extra technical work, one could obtain a lower bound for the number of sign changes growing as a small power of $\log T$.  As it stands, Corollary \ref{coro:signchanges} implies that for each $n$, there exists $T$ large enough (compared to $n$) so that $E_T^*(x+iy)$ changes sign at least $n$ times for $\alpha < y < \beta$.  This follows by cutting the interval $(\alpha, \beta)$ into $n$ disjoint subintervals, and applying Corollary \ref{coro:signchanges} to each one.

\subsection{Conjectures for Maass cusp forms}
The techniques used in this paper are, by and large, not applicable to Maass cusp forms.  One should keep in mind that the restricted QUE problem is more difficult than the ``standard'' QUE conjecture, which in turn is  known, but with no rate of convergence (however, the generalized Lindel\"of hypothesis would give an optimal rate of convergence).  Nevertheless, it is natural to conjecture an analog of Theorem \ref{thm:QUE} for Maass forms (the case $x=0$ was conjectured in \cite{Young}).  Let $u_j$ be a Hecke-Maass cusp form, with Laplace eigenvalue $1/4 + T_j^2$.  
Suppose that $\lambda_{u_j}(m)$ is the $m$-th Hecke eigenvalue of $u_j$, and that $u_j(-x+iy) = \delta_j u_j(x+iy)$, where $\delta_j = \pm 1$.  Moreover, normalize $u_j$ so that $\int_{\Gamma \backslash \mh} |u_j(z)|^2 \frac{dx dy}{y^2}  = \vol(\Gamma \backslash \mh) = \frac{\pi}{3}$.
\begin{myconj}
\label{conj:QUEcuspCase}
 Let conditions be as in Theorem \ref{thm:QUE}.  Then as $T_j \rightarrow \infty$, we have
 \begin{equation}
 \label{eq:QUEcuspCase}
\int_0^{\infty} \psi^2(y) |u_j(x+iy)|^2 \frac{dy}{y} \sim
\int_0^{\infty} \psi^2(y) \Big[1 + \delta_j B_q(1,u_j) J_0\Big(\frac{\theta T_j}{y}\Big)\Big] \frac{dy}{y}.
 \end{equation}
 Here $B_q(s,u_j)$ is a finite Euler product given by \eqref{eq:BqDefCuspFormCase}.
 \end{myconj}
Remarks.  We refrain from guessing an explicit error term, but we expect that the asymptotic is uniform in $x$.
We also extend a conjectural version of Corollary \ref{coro:signchanges}.  There is a small adjustment needed to cover odd Maass forms, which all vanish along the lines $x \in \mz$ and $x \in \frac12 + \mz$.  We therefore need to restrict $x$ to not lie within a distance $O(T^{-1} (\log T)^{\varepsilon})$ of one of these lines.

%
Suppose that $M$ is a compact Riemannian manifold with ergodic geodesic flow. In \cite{DZ} and \cite{TothZelditch}, it is proven that for a generic compact geodesic segment $\gamma \subset M$, 
\begin{equation}\label{eq}
\lim_{j \to \infty} \|u_j\|_{L^2(\gamma)}^2 =  \mathrm{length} (\gamma)
\end{equation}
holds for a density $1$ subsequence of eigenfunctions. Until now, the only known example where \eqref{eq} fails is when there exists a symmetry on the ambient manifold $M$ that fixes $\gamma$. In this case, approximately half of the eigenfunctions vanish identically on $\gamma$, and the other half are expected to satisfy
\begin{equation*}
\lim_{j \to \infty} \|u_j\|_{L^2(\gamma)}^2 =  2 \thinspace \mathrm{length} (\gamma).
\end{equation*}
Conjecture \ref{conj:QUEcuspCase} indicates that, for any rational value of $2x = \frac{a}{q}$ with $q > 1$, the limit $\lim_j \| u_j \|_{L^2(\gamma)}$ does not exist.  The reason for the lack of convergence is that the quantities $B_q(1,u_j)$ fluctuate with $u_j$.  Analogously to \eqref{eq:Bp1T}, if $q=p$ is prime then
\begin{equation}
 B_p(1,u_j) = \frac{1}{p+1} (1 +  \lambda_{u_j}(p) -p^{-1})
\end{equation}
The distribution of the values of $\lambda_{u_j}(p)$, with $p$ a fixed prime, was determined by Sarnak \cite{SarnakPlancherel}, and is given by the $p$-adic Plancherel measure.  As a consequence, for any interval $(a,b) \subset [-2,2]$, there exists a \emph{positive-density} subsequence of $u_j$'s with $\lambda_{u_j}(p) \in (a,b)$.  In fact, Sarnak proved a joint equidistribution statement for the values $\lambda_{u_j}(p_1), \dots, \lambda_{u_j}(p_k)$ for any finite set of primes $p_1, \dots, p_k$, which may be used to constrain the values of $B_q(1,u_j)$ for any $q$.

\subsection{Overview}
We gather some notation  from \cite{Young}.  
Let $E_T^*(z) = \frac{\theta(1/2 +
iT)}{|\theta(1/2 + iT)|} E_T(z)$, where $\theta(s) = \pi^{-s} \Gamma(s)
\zeta(2s)$, so $E_T^*$ is real-valued for $z \in \mh$.
The Fourier expansion of $E_T^*$ is
\begin{equation}
\label{eq:ETFourier}
 E_T^*(x+iy) = \mu y^{1/2 + iT} + \overline{\mu} y^{1/2-iT} + \rho^*(1) \sum_{n
\neq 0} \frac{\tau_{iT}(n) e(nx)}{|n|^{1/2}} V_T(2 \pi |n| y).
\end{equation}
Here $\rho^*(1) = (2/\pi)^{1/2} |\theta(1/2 + iT)|^{-1}$,
$\mu =
\frac{\theta(1/2 + iT)}{|\theta(1/2 + iT)|}$, $V_T(y) = \sqrt{y} K_{iT}(y)$, and $\tau_{iT}(n) = \sum_{ab = |n|} (a/b)^{iT}$.

There are two main differences in the generalization from $x=0$ (as in \cite{Young}) to arbitrary $x$.
The most important difference is that when $x=0$, the coefficients $e(nx) \tau_{iT}(n)$ of $E_T^*$ are multiplicative, while for general $x \in \mr$ they are not.  This causes some crucial aspects of \cite{Young} to break down.
 The other significant difference comes from an asymmetry between $n > 0$ and $n < 0$ in \eqref{eq:ETFourier}.  Note that the $K$-Bessel function in the Fourier expansion only depends on $|n|$, which has the following practical effect.   When one considers a $y$-integral of $|E_T^*(x+iy)|^2$, analyzed by squaring out and integrating the Fourier expansion term-by-term, there will be diagonal terms and off-diagonal terms.  The diagonal parts will include a sum with say $n_1 = n_2$ as well as terms with $n_1 = -n_2$, the latter of which will then lead to sums of the form $\sum_n |\tau_{iT}(n)|^2 e(2nx)$.  It is also necessary to treat the opposite sign off-diagonal terms of the form $\sum_{h \neq 0} \sum_{n} \tau_{iT}(n) \tau_{iT}(n+h) e((2n + h)x)$.
 It is difficult to estimate such sums because the weight is highly oscillatory (although, not for $x=0$, for instance). 

To continue the discussion, we describe the earliest stages of analysis of $I_{\psi}(x,T)$.  By a use of Parseval's formula (in the Mellin transform setting), $I_{\psi}(x,T)$ decomposes as
\begin{equation}
\label{eq:IpsiFitSquared}
 I_{\psi}(x,T) = \frac{1}{2 \pi} \intR |F(it)|^2 dt,
\end{equation}
where
\begin{equation}
\label{eq:Fdef}
 F(s) = F_{\psi, x, T}(s)  
 = \int_0^{\infty} \psi(y) y^s E_T^*(x+iy) \frac{dy}{y}.
\end{equation}
The size of $\text{Im}(s)$ is the most important basic parameter in the
analysis of $I_{\psi}(x,T)$.

We will then insert the Fourier expansion \eqref{eq:ETFourier} and calculate the $y$-integral in terms of gamma functions (see \eqref{eq:gammaVformula} below).  One arrives at an integral of the form
\begin{equation*}
|\zeta(1+2iT)|^{-2} \int_{-T}^{T} (1 + (T^2-t^2))^{-1/2} \Big| \sum_{n \neq 0} 
\frac{\tau_{iT}(n)
e(nx)}{|n|^{1/2 + it}} k(n,t,T) \Big|^2 dt,
\end{equation*}
as well as some other terms of lesser importance.  Here $k$ is an innocuous weight function, which restricts the support of $n$ to $|n| \asymp  (1 + (T^2- t^2))^{1/2}$.  The integral naturally decomposes into pieces of the form $T -  |t| \asymp \Delta$ where $1 \ll \Delta = o(T)$, as well as the ``bulk" range with $|t| \leq T(1-o(1))$.  The behavior of $F(it)$ depends strongly on the size of $\Delta$, and our method of proof breaks up into cases.  

The mean value theorem for Dirichlet polynomials (see Lemma \ref{lemma:MontgomeryVaughan} below) leads to a bound of $O(\log T)$ for \emph{each} dyadic segment $T\mp t \asymp \Delta$.  Since there are $O(\log T)$ such dyadic segments, this leads to the ``trivial'' bound $I_{\psi}(x,T) \ll \log^2 T$.

The main term in Theorem \ref{thm:QUE} comes from the bulk range, $|t| \leq T - o(T)$.  This main term arises, after squaring out, from the diagonal terms $m=n$ and the opposite diagonal terms, $m=-n$.  In the bulk range, the mean value theorem gives an upper bound of the same order of magnitude as the main term, which means that \emph{any} savings in the off-diagonal terms will be sufficient towards Theorem \ref{thm:QUE}.  
For this reason, we apply Holowinsky's sieve method for bounding shifted convolution sums with absolute values.  
This dodges any problems with estimating the shifted convolution sum $\sum_n \tau_{iT}(n) \tau_{iT}(n+h) e((2n+h)x)$ with oscillatory weight $e((2n+h)x)$, but still requires an asymptotic for $\sum_n |\tau_{iT}(n)|^2 e(2nx)$.
This approach leads to the following
\begin{myprop}
\label{prop:bulkrangeIntro}
 Let $\eta = \eta(T)$ satisfy $(\log T)^{-\delta} \ll \eta < \delta$, for some $0 < \delta < 2/3$.  Then
 \begin{multline}
  \int_{|t| \leq T - \eta(T) T} |F(it)|^2 \frac{dt}{2\pi} = \frac{3}{\pi} \log(1/4 + T^2) \int_0^{\infty} \psi^2(y) \Big[1 + B_q(1,T) J_0\Big(\frac{\theta T}{y}\Big)\Big] \frac{dy}{y} 
\\  
  + O(\eta^{-3/2} (\log T)^{8/9 + \varepsilon} + \eta^{1/2} \log T).
 \end{multline}
\end{myprop}
Proposition \ref{prop:bulkrangeIntro} implies a sharp \emph{lower bound} on $I_{\psi}(x,T)$,  using the positivity in \eqref{eq:IpsiFitSquared}.

Holowinsky's method of bounding the off-diagonal terms with absolute values is able to save, at best, a small fractional power of $\log T$.  Since the trivial bound on $I_{\psi}(x,T)$ is $O(\log^2 T)$, while the main term is of size $\log T$, this method unfortunately cannot be used to completely prove Theorem \ref{thm:QUE} (since one needs to save an entire factor of $\log T$).  
This is a qualitative difference between QUE and restricted QUE.

Having accounted for the main term with Proposition \ref{prop:bulkrangeIntro}, it suffices to \emph{bound from above} the contribution from the remaining ranges of $t$.  For this reason, we can apply the simple inequality $|\sum_{n \neq 0} a_n|^2 \leq 2 |\sum_{n > 0} a_n|^2 + 2 |\sum_{n < 0} a_n|^2$ to $|F(it)|^2$.  This has the significant effect of avoiding all the difficulties with the terms with opposite sign.  The problem then reduces to understanding $\sum_{n} \tau_{iT}(n) \tau_{iT}(n+h)$ which in certain ranges we can quote earlier work \cite{Young}.  This leads to the following
\begin{myprop}
\label{prop:biggisht}
Suppose that $\delta > \eta(T) \gg (\log T)^{-\delta}$ for some fixed small  $\delta > 0$, and
\begin{equation}
T^{\frac{25}{27} + \varepsilon} \ll \Delta \leq \eta(T) T.
\end{equation}
Then for some $\delta' > 0$, we have 
\begin{equation}
\label{eq:biggisht}
 \int_{\frac12 \Delta \leq T- |t| \leq \Delta} |F(it)|^2 dt \ll \frac{\Delta^{1/2}}{T^{1/2}} \log
T + T^{-\delta'} + S_{\Delta},
\end{equation}
where $S_{\Delta}$ is a nonnegative quantity satisfying
\begin{equation}
\label{eq:SDeltaProperty}
 \sum_{\substack{1 \ll \Delta \ll T \\ \text{dyadic}}} S_{\Delta} \ll 1.
\end{equation}
\end{myprop}
 Here $S_{\Delta}$ arises as an off-diagonal main term that is surprisingly difficult to estimate.  
 Note that Proposition \ref{prop:biggisht} implies
 \begin{equation}
 \label{eq:RandomSum}
  \int_{T^{25/27+\varepsilon} \ll T-|t| \ll \eta T} |F(it)|^2 dt \ll \eta^{1/2} \log T + 1.
 \end{equation}
 For a fixed $\Delta$, the trivial bound is $S_{\Delta} \ll 1$, so the bound \eqref{eq:SDeltaProperty} amounts to a saving of $\log T$ \emph{on average} over the dyadic segments, since there are $O(\log T)$ such dyadic segments.  
One sees that the bound \eqref{eq:SDeltaProperty} is crucial, since the individual bound $S_{\Delta} \ll 1$ would only lead to a bound of $\log T$ in \eqref{eq:RandomSum}.

When $\Delta$ is somewhat smaller than $T$, which means that $N = \sqrt{\Delta T}$ is somewhat larger than $\Delta$,
the problem may be interpreted as estimating the mean value of a ``long'' Dirichlet polynomial.  The shifted divisor sum approach from \cite{Young} breaks down for $\Delta \ll T^{25/27}$, and different techniques are needed to cover all the sizes of $\Delta$.

An approach to the problem is to open up the divisor function $\tau_{iT}(n) = \sum_{ab=n} (a/b)^{iT}$, and use the sums over $a$ and $b$ asymmetrically.
Arguments in this vein appear in Section \ref{section:smallmediumDelta}.
By a dyadic partition of unity, and Cauchy's inequality, we have
\begin{equation}
\label{eq:Section1.2Formula}
\int_{T - |t| \asymp \Delta} \Big| \sum_{n \asymp N} \frac{\tau_{iT}(n) e(nx)}{n^{1/2+it}} \Big|^2 \ll \log T \sum_{B \text{ dyadic}} \int_{T-|t| \asymp \Delta} \Big| \sum_{a \asymp A} \sum_{b \asymp B} \frac{ e(abx)}{a^{1/2+it-iT} b^{1/2+it+iT}} \Big|^2,
\end{equation}
where $AB = N$ (observe that once $b$ is localized by $b \asymp B$, then automatically $a \asymp N/B$). 
We use a variety of methods depending on the sizes of $A$ and $B$.

One fruitful option is to apply Poisson summation in either $a$ or $b$, prior to any analysis of the $t$-integral.
It turns out that if $B \gg T^{1/2+\varepsilon}$, then one can get a savings from Poisson summation on $b$ inside the square (and prior to any $t$-integration), which can then be combined with savings from the $t$-integral.  See Lemma \ref{lemma:bigB} for the result obtained by this idea.
Similarly, if $A \gg \Delta^{1/2+\varepsilon}$, then there is savings from Poisson on the $a$-sum, for which see Lemma \ref{lemma:bigA}.
Since $AB = N = (\Delta T)^{1/2}$, this means $A = \Delta^{1/2+o(1)}$ and $B = T^{1/2+o(1)}$ is the main region of interest.  

Another approach is to open the square and evaluate the $t$-integral, leading to a shifted convolution sum of the form
\begin{equation}
\label{eq:IntroShiftedSumFormula}
\Delta
\sum_{|h| \ll \frac{AB}{\Delta}} e(hx)
\sum_{a_1 b_1 - a_2 b_2 = h} \frac{  1}{(a_1 a_2 b_1 b_2)^{1/2} (b_1/b_2)^{2iT}}.
\end{equation}
Duke, Friedlander, and Iwaniec \cite{DFI} \cite{DFIbilinear} have studied general sums of this form, and obtained satisfactory results in the case that $A \gg B^{1-\delta}$ for some (small) $\delta > 0$ (note also there is recent work of Bettin and Chandee \cite{BettinChandee} that can improve the size of $\delta$).  Based on the locations of $A$ and $B$ established from the previous paragraph, we see that this is satisfactory only for $\Delta \gg T^{1-\delta'}$ for some (small) $\delta' > 0$, yet we require all $\Delta \ll T^{25/27+\varepsilon}$.  The first step in the method of \cite{DFI} is to solve the equation $a_1 b_1 - a_2 b_2 = h$ by interpeting it as a congruence $a_1 \equiv - h \overline{b_1} \pmod{b_2}$, and then applying Poisson summation in $a_1 \pmod{b_1}$.  Since $B$ is larger than $A$, one obvious approach is to reverse the roles of the $a_i$ and the $b_i$ in the \cite{DFI} approach.  The problem with this is that the oscillatory factor $(b_1/b_2)^{-2iT}$ causes a loss of efficiency in Poisson summation (the trivial bound on the dual sum is worse than the trivial bound on the original one).  
There is a way to partially circumvent the problem arising from this oscillation, which is to note that the equation $a_1 b_1 - a_2 b_2 = h$ means that $b_1/b_2$ is approximately $a_2/a_1$.  Precisely,
\begin{equation}
\label{eq:OscillationSwitching}
 (b_1/b_2)^{2iT} = (a_2/a_1)^{2iT} \Big(1+\frac{h}{a_2 b_2}\Big)^{2iT}.
\end{equation}
In terms of $b_2$, this is much less oscillatory than before, since
\begin{equation*}
 \Big(1+\frac{h}{a_2 b_2}\Big)^{2iT} \approx \exp\Big(\frac{2iT h}{a_2 b_2} \Big), \quad \text{and} \quad \frac{Th}{a_2 b_2} \ll \frac{T}{\Delta}.
\end{equation*}
After these initial algebraic manipulations, we may then solve the equation $a_1 b_1 - a_2 b_2 = h$ by $b_2 \equiv - \overline{a_2} h \pmod{a_1}$.  It is better to interpet this equation as a congruence modulo $a_1$ instead of $b_1$, because $A$ is smaller than $B$, and moreover, $b_1$ was already eliminated from the right hand side of \eqref{eq:OscillationSwitching}.  The Poisson summation argument leads to a dual sum which when bounded trivially is just barely not satisfactory (at least, in the relevant ranges $A = \Delta^{1/2+o(1)}$, $B= T^{1/2+o(1)}$).  However, the gain is that the dual sum has a phase (still of rough magnitude $\frac{Th}{AB}$, but no longer linear in $h$).  We may thus obtain some cancellation in the $h$-sum using classical bounds for exponential sums.  See Lemma \ref{lemma:ShiftedConvolutionWithBsummed} for the result using this idea.

There is also a technical problem to treat small values of $b$.  Heuristically, if $b$ is large then there should be significant cancellation from the $b$-sum appearing in \eqref{eq:Section1.2Formula}, but for $b$ small this may not be true.  It turns out that these terms lead to the same type of off-diagonal main term that was discussed earlier in the context of the shifted convolution problem with $\Delta \gg T^{25/27+\varepsilon}$.  Our approach to these terms is to solve the shifted convolution sum in \eqref{eq:IntroShiftedSumFormula} by Poisson summation in $a$ modulo $b_1$, as in \cite{DFI}, for which see Lemma \ref{lemma:smallB}.

Taken together, the above ideas lead to the following
\begin{myprop}
\label{prop:FboundmediumDelta}
Fix $0 < \delta < 1/4$.  Then
\begin{equation}
\int_{T^{\delta} \ll T-|t| \ll T^{1-\delta}} |F(it)|^2 dt \ll_{\delta} 1.
\end{equation}
\end{myprop}

For technical reasons, the very smallest values of $|T \mp t|$ require a slightly different proof, and we claim this estimate with
\begin{myprop}
\label{prop:verysmallDelta}
We have
\begin{equation}
\int_{|T \mp t| \ll T^{1/100}} |F(it)|^2 dt \ll_{\varepsilon} (\log T)^{\varepsilon}.
\end{equation}
\end{myprop}
 
 Propositions \ref{prop:bulkrangeIntro}, \ref{prop:biggisht}, \ref{prop:FboundmediumDelta}, and \ref{prop:verysmallDelta}, along with an easy fact that the contribution from $|t| \geq T + T^{1/100}$ is very small (for which see Section \ref{section:initialdevelopments}), altogether imply Theorem \ref{thm:QUE}, taking $\eta = (\log T)^{-1/18}$.  The proof of Proposition \ref{prop:bulkrangeIntro} appears in Sections \ref{section:largeDelta} and \ref{section:spectralapproach}.  The main part of this paper is in proving Proposition \ref{prop:FboundmediumDelta}, which appears in Section \ref{section:smallmediumDelta}.  Finally, we show Proposition \ref{prop:verysmallDelta} in Section \ref{section:verysmallDelta}.

 The proof of Corollary \ref{coro:signchanges} appears in Section \ref{section:shiftedQUE}, and discussion of Conjecture \ref{conj:QUEcuspCase} occurs in Section \ref{section:conjecture}.

\subsection{Remarks on Conjecture \ref{conj:QUEcuspCase}}
 The outline of Theorem \ref{thm:QUE} given in the previous subsection does not apply to Conjecture \ref{conj:QUEcuspCase}, for a number of reasons.  One important reason is that the proofs of Propositions \ref{prop:bulkrangeIntro} and \ref{prop:biggisht} rely on a solution of the shifted convolution problem which  is intertwined with the QUE problem on $\Gamma \backslash \mh$ (see \cite[Appendix A]{GRS} for more discussion).  We currently do not have a quantitative QUE theorem (e.g. with power-saving error term).  Another problem is that the proof of Proposition \ref{prop:FboundmediumDelta} opens the divisor function as the Dirichlet convolution of smooth functions,   which is not available for general Maass cusp forms.

\subsection{Acknowledgements}
I thank Andrew Booker, Junehyuk Jung, Eren Mehmet K\i ral, and K. Soundararajan for  valuable input on this work, and Peter Sarnak for a question that motivated this paper. 
I also thank the referee for valuable comments and corrections.
Some of this work was performed while the author was a member at the Institute for Advanced Study in Fall 2014.

\section{Notation and initial developments}
\subsection{Notation}
Recall $V_T(y) = \sqrt{y} K_{iT}(y)$.
By a well-known formula for the Mellin transform of the $K$-Bessel function, we
have
\begin{equation}
\label{eq:gammaVformula}
 \gamma_{V_T}(1/2 + s) := \int_0^{\infty} V_T(2 \pi y) y^s \frac{dy}{y} =
2^{-3/2} \pi^{-s} \Gamma\Big(\frac{1/2 + s + iT}{2}\Big) \Gamma\Big(\frac{1/2 +
s - iT}{2}\Big).
\end{equation}
It is also helpful to recall the Rankin--Selberg $L$-function
\begin{equation}
\label{eq:Zformula}
 Z(s,E_T) := \sum_{n=1}^{\infty} \frac{\tau_{iT}(n)^2}{n^s} = \frac{\zeta(s-2iT)
\zeta(s+2iT) \zeta(s)^2}{\zeta(2s)}.
\end{equation}

\subsection{Miscellaneous lemmas}
Here we collect some basic tools used throughout this paper.

\begin{mylemma}[Vinogradov--Korobov]
\label{lemma:VinogradovKorobov}
 For some $c > 0$ and for any large $|t| \gg 1$, $1- \frac{c}{(\log |t|)^{2/3}} \leq \sigma \leq 1$, we have
 \begin{equation}
 \label{eq:VinogradovKorobovZeta}
  \zeta(\sigma + it)^{\pm 1} \ll (\log |t|)^{2/3+\varepsilon},
 \end{equation}
and
\begin{equation}
\label{eq:VinogradovKorobovZetalogarithmicderivative}
 \frac{\zeta'}{\zeta}(1 + it) \ll (\log |t|)^{2/3 + \varepsilon},
\end{equation}
\end{mylemma}
For a reference, see \cite[Corollary 8.28, Theorem 8.29]{IK}.
As a corollary to Lemma \ref{lemma:VinogradovKorobov}, it follows from standard methods from \cite{Davenport} (see \cite[Lemma 2.1]{GranvilleSound} for more details) that
\begin{equation}
\label{eq:logzetaapproximation}
\sum_{p \leq x} p^{-1-it} = \log \zeta(1+it) + O(1),
\end{equation}
provided $\log x \gg (\log |t|)^{2/3+\varepsilon}$.

We also need the mean value theorem for Dirichlet polynomials of \cite{MontgomeryVaughan}.
\begin{mylemma}[Montgomery-Vaughan]
\label{lemma:MontgomeryVaughan} 
Let $a_n$ be an arbitrary complex sequence. We have
 \begin{equation*}
  \int_0^{U} \Big|\sum_{n \leq N} a_n n^{-iu} \Big|^2 du = \sum_{n \leq N} |a_n|^2 (U + O(n)).
 \end{equation*}
 More generally, if $\lambda_1, \dots, \lambda_R$ are distinct real numbers, and $\delta_r = \min_{s \neq r} |\lambda_r - \lambda_s|$, then
 \begin{equation}
 \label{eq:MontgomeryVaughanVersion2}
  \int_0^{U} \Big| \sum_{r \leq R} a_r e^{i \lambda_r u} \Big|^2 du \leq \sum_{r \leq R} |a_r|^2 (U + 3\pi \delta_r^{-1}) .
 \end{equation}
\end{mylemma}

We require the following change of basis formula from additive characters to multiplicative.
\begin{mylemma}
\label{lemma:changeofbasis}
Suppose that $c_n$ is a finite sequence of complex numbers, $q \geq 1$ is an integer, and $(a,q)=1$.  Then
\begin{equation}
\sum_n c_n e\Big(\frac{an}{q}\Big) = \sum_{d|q} \frac{1}{\phi(q/d)} \sum_{\chi \shortmod{q/d}} \tau(\overline{\chi}) \chi(a) \sum_n c_{dn} \chi(n).
\end{equation}
\end{mylemma}
The proof follows from \cite[(3.11)]{IK}, letting $d=(n,q)$.

We will also have need of the following
\begin{myprop}[Jutila-Motohashi \cite{JutilaMotohashiCrelle}]
\label{prop:JM}
Let $1 \ll U \ll V$, and let $q$ be a positive integer.  Then
 \begin{equation}
\int_{V \leq |u| \leq V + U} \sum_{\chi \shortmod{q}} |L(1/2 +iu, \chi)|^4 du
\ll q^{1+\varepsilon} (U + V^{2/3})^{1+\varepsilon},
\end{equation}
\end{myprop}
This is a generalization of a result of Iwaniec \cite{IwaniecFourthMoment}, corresponding to the case $q=1$.

We will also frequently refer to results in the literature on exponential integrals and exponential sums.  Textbook references include  \cite{GK} \cite{Huxley} \cite{IK}.

\subsection{Initial developments}
\label{section:initialdevelopments}
We will calculate $F_{\psi,x,T}(s) = F(s)$ here, and then perform some easy approximations.  Inserting \eqref{eq:ETFourier} into \eqref{eq:Fdef}, using \eqref{eq:gammaVformula} and the Mellin inversion formula $\psi(y) = \frac{1}{2 \pi i} \int_{(1)} \widetilde{\psi}(-u) y^{u} du$, we obtain
\begin{multline}
\label{eq:Fformula1}
 F(s) = \mu \widetilde{\psi}(1/2 + s + iT) + \overline{\mu}
\widetilde{\psi}(1/2 + s - iT)
\\
+ \frac{\rho^*(1)}{(2 \pi )^{s}} \sum_{n \neq 0} \frac{\tau_{iT}(n)
e(nx)}{|n|^{1/2 + s}} \frac{1}{2 \pi i} \int_{(1)} \widetilde{\psi}(-u) |n|^{-u}
\gamma_{V_T}(1/2 + s+u) du.
\end{multline}
Next we develop some simple approximations.  First, note that
\begin{equation}
\label{eq:rhosize}
 \frac{|\rho^*(1)|^2}{\cosh(\pi T)} = \frac{2/\pi}{|\zeta(1+2iT)|^2}.
\end{equation}
By \eqref{eq:gammaVformula} and Stirling's bound,
combined with \eqref{eq:rhosize}, the integral appearing in
\eqref{eq:Fformula1} is exponentially small if $|\text{Im}(s+u)-T| \gg
(\log^2 T)$.  Since $\widetilde{\psi}$ has rapid decay in the imaginary
direction, we have that $F(it) \ll (1+|t|)^{-2} T^{-100}$, say, for $|t| \geq
T + T^{\varepsilon}$.

We can simplify $F(s)$ in the range $|\text{Im}(s)| \leq T -
T^{\varepsilon}$.  Let
\begin{equation}
\label{eq:F0def}
 F_0(s) = \rho^*(1) \gamma_{V_T}(1/2+s)
\sum_{n \neq 0} \frac{\tau_{iT}(n)
e(nx)}{|n|^{1/2 + s}} \psi
\Big(\frac{\sqrt{|s+iT||s-iT|}}{2 \pi |n|} \Big).
\end{equation}
\begin{mylemma}
\label{lemma:FitF0approximation}
 For each $\varepsilon > 0$ there exists $\delta > 0$ so that
 \begin{equation}
  \int_{|t| \leq T - T^{\varepsilon}} |F(it)|^2 dt =
  \int_{|t| \leq T - T^{\varepsilon}} |F_0(it)|^2 dt + O(T^{-\delta}).
 \end{equation}
\end{mylemma}
As a consequence of Lemma \ref{lemma:FitF0approximation}, 
in the statements of Propositions \ref{prop:bulkrangeIntro}, \ref{prop:biggisht}, and \ref{prop:FboundmediumDelta}, we may freely replace $F$ by $F_0$.  We will do so systematically in the rest of the paper, with no further mention.

Note that if $s= it$, $0 \leq t \leq T$, the function $\psi$ appearing in
\eqref{eq:F0def} is supported on $|n| \asymp T^{1/2} (T-t)^{1/2}$.  
 
\begin{proof}
In the range $|t|  \leq T - T^{\varepsilon}$, Stirling's formula gives
\begin{equation*}
 \Gamma\Big(\frac{1/2 + s + u \pm iT}{2}\Big) \sim \Gamma\Big(\frac{1/2 + s \pm iT}{2}\Big)
 \Big(\frac{1/2 + s \pm iT}{2} \Big)^{u/2} \Big(1 + \sum_{j=1}^{\infty}
\frac{P_{j, \pm}(u)}{(s \pm iT)^j} \Big),
\end{equation*}
where $P_{j, \pm}$ is a polynomial, and
where the meaning of $\sim$ here is as an asymptotic expansion.  Therefore,
since $|s \pm iT| \gg T^{\varepsilon}$, we can truncate the expansion at $j
\asymp \varepsilon^{-2}$ with a very small error term.
By \eqref{eq:gammaVformula}, we have
\begin{equation*}
\gamma_{V_T}(1/2 + s + u) \sim (2\pi)^{-u} |s+iT|^{\frac{u}{2}} |s-iT|^{\frac{u}{2}}
\gamma_{V_T}(1/2 + s) 
\prod_{\pm}
\Big(1 + \sum_{j=1}^{\infty}
\frac{P_{j,\pm}(u)}{(s \pm iT)^j} \Big).
\end{equation*}
Therefore, for $s=it$, $|t| \leq T - T^{\varepsilon}$,
\begin{multline*}
F(s) = \Big[\rho^*(1) \gamma_{V_T}(1/2+s)
\sum_{n \neq 0} \frac{\tau_{iT}(n)
e(nx)}{|n|^{1/2 + s}}
\\
\frac{1}{2 \pi i} \int_{(1)} \widetilde{\psi}(-u)
\Big(\frac{\sqrt{|s+iT||s-iT|}}{2 \pi |n|} \Big)^u
\prod_{\pm} \Big(1 + \sum_{1 \leq j \leq \epsilon^{-2}} \frac{P_{j, \pm}(u)}{(s \pm
iT)^j} \Big) du \Big] + O(T^{-100}).
\end{multline*}
The leading term in the $u$-integral is calculated exactly via \eqref{eq:F0def}, so that
\begin{equation*}
F(s) =\sum_{0 \leq l \leq \epsilon^{-2}} F_l(s) + O(T^{-100}),
\end{equation*}
say.
The lower-order terms in the asymptotic expansion with $l \geq 1$ involve linear combinations
of derivatives of $\psi$ in place of $\psi$, and multiplied by factors of the
form $(s + iT)^{-j_1} (s-iT)^{-j_2}$, and are hence all of the same essential
shape yet smaller by these additional factors.  In fact, one can show that
$\int_{|t| \leq T - T^{\varepsilon}} |F_0(it)|^2 dt \ll T^{\varepsilon'}$, using only the mean value theorem for
Dirichlet polynomials (Lemma \ref{lemma:MontgomeryVaughan}), and by similar reasoning, for $j \geq 1$, we have $\int_{|t|
\leq T-T^{\varepsilon}} |F_j(it)|^2 dt \ll T^{-\delta}$, for some $\delta > 0$.
\end{proof}

Next we smoothly decompose the $t$-integral into pieces of the form $T - |t|
\asymp \Delta$ where $T^{\varepsilon} \ll \Delta \ll T$ (it will also be
necessary to treat the integral of $|F(it)|^2$ with $T-|t| \ll
T^{\varepsilon}$ for which see Section \ref{section:verysmallDelta} below).
As a first-order
approximation, we mention that for such $t$, we have
\begin{equation}
\label{eq:gammaVStirling}
t^k \frac{d^k}{dt^k} \cosh(\pi T) |\gamma_{V_T}(1/2 + it)|^2 \ll 
(\Delta T)^{-1/2}, \quad k=0,1,2,\dots.
\end{equation}

For the purposes of calculating diagonal terms, it is convenient to record the
following
\begin{mylemma}
\label{lemma:diagonaltermsum}
Let $w$ be a fixed smooth, compactly-supported function.
Suppose $\log N \gg (\log T)^{2/3 + \delta}$, for some fixed $\delta >
0$.  Then
\begin{equation}
\label{eq:diagonaltermsum}
 \sum_{n=1}^{\infty} |\tau_{iT}(n)|^2 w\Big(\frac{n}{N}\Big) = \frac{|\zeta(1 +
2iT)|^2}{\zeta(2)}  [\widetilde{w}(1) N \log{N} + O(N(\log T)^{2/3 +
\varepsilon}) ]. 
\end{equation}
\end{mylemma}
\begin{proof}
 By Mellin inversion and \eqref{eq:Zformula}, the sum on the left hand side of
\eqref{eq:diagonaltermsum} is
\begin{equation*}
 \frac{1}{2 \pi i} \int_{(2)} N^s \widetilde{w}(s) \frac{\zeta(s-2iT)
\zeta(s+2iT) \zeta(s)^2}{\zeta(2s)}  ds.
\end{equation*}
We evaluate the integral by the standard method of moving the contour, in this case to
$\text{Re}(s) = \sigma = 1- \frac{c}{(\log T)^{2/3}}$, for some constant $c > 0$.  We first calculate the residue at $s=1$. Using
\begin{equation*}
\begin{gathered}
 N^s = 1 + (s-1) \log N + O((s-1)^2),
 \\
 \zeta(s-2iT)\zeta(s+2iT) = |\zeta(1+2iT)|^2 (1 + (s-1)
2 \text{Re}\frac{\zeta'}{\zeta}(1 + 2iT) + O((s-1)^2),
\end{gathered}
\end{equation*}
and the Vinogradov--Korobov bound (Lemma \ref{lemma:VinogradovKorobov}),
one calculates that the residue at $s=1$ is consistent with the right hand side of
\eqref{eq:diagonaltermsum}.  The residues at $s = 1 \pm 2iT$ are negligible because $\widetilde{w}$ has rapid decay in the imaginary direction.

Finally we bound the integral along the line $\sigma$.  We use the Vinogradov--Korobov bound again, which leads to an error of size 
$ N (\log T)^{4/3} \exp(-c \frac{\log N}{(\log T)^{2/3}})$, which is bounded by the error term appearing in \eqref{eq:diagonaltermsum}.
\end{proof}

We  directly quote the following.
\begin{myprop}[\cite{Young}, Theorem 8.1]
\label{prop:SCS}
Suppose that $w(x)$ is a smooth function on the positive reals supported on $Y \leq x \leq  2Y$ and satisfying $w^{(j)}(x) \ll_j (P/Y)^j$ for some parameters $1 \leq P \leq Y$.  Let $\theta = 7/64$, and set $R = P + \frac{T|m|}{Y}$.  Then for $m \neq 0$, $ R \ll T/(TY)^{\delta}$, we have
\begin{equation}
\label{EisensteinShiftedConvolutionSum}
 \sum_{n \in \mz} \tau_{iT}(n) \tau_{iT}(n+m) w(n) = M.T. +  E.T.,
\end{equation}
where
\begin{equation}
\label{eq:MTdef}
 M.T. = \sum_{\pm}  \frac{|\zeta(1 + 2iT)|^2}{\zeta(2)} \sigma_{-1}(m) \int_{\max(0, -m)}^{\infty} (x + m)^{\mp iT} x^{\pm iT} w(x) dx,
\end{equation}
and
\begin{equation}
\label{eq:ETdef}
 E.T.\ll 
(|m|^{\theta} T^{\frac13} Y^{\frac12} R^2 
+ 
T^{\frac16} Y^{\frac34}  R^{\frac12} ) (TY)^{\varepsilon} .
\end{equation}
Furthermore, with $R = P + \frac{TM}{Y}$, we have
\begin{equation}
\label{eq:ETsummedoverm}
 \sum_{1 \leq |m| \leq M} |E.T.| \ll (M T^{\frac13} Y^{\frac12} R^2 
+ 
M T^{\frac16} Y^{\frac34}  R^{\frac12} ) (TY)^{\varepsilon}.
\end{equation}
\end{myprop}
Remark.  The bound \eqref{eq:ETsummedoverm} roughly means that the Ramanujan--Petersson conjecture, i.e., $\theta=0$, holds on average over $m$.

\section{The bulk range: extracting the main term}
\label{section:largeDelta}
The purpose of this section is to show Proposition \ref{prop:bulkrangeIntro}.

\subsection{First steps}
For shorthand, set $\eta = \eta(T)$.
Let $W(t)$ be a nonnegative function that is $1$ on $|t| \leq T - c_1
\eta T$, and $0$ on $|t| \geq T - c_2 \eta T$,
where $W^{(j)}(t) \ll (\eta T)^{-j}$, and where
$c_1 > c_2 > 0$ are fixed positive constants.

We have
\begin{equation}
\label{eq:F0integral}
 \intR W(t) |F_0(it)|^2 \frac{dt}{2\pi} = \frac{|\rho^*(1)|^2}{\cosh (\pi T)} \sum_{m, n}
\frac{\tau_{iT}(m) \tau_{iT}(n) e((n-m)x)}{|mn|^{1/2}} J(m,n),
\end{equation}
where
\begin{equation}
\label{eq:Jdef}
 J(m,n) = \intR W(t)
|\gamma_{V_T}(1/2 + it)|^2 \cosh(\pi T)
\psi
\Big(\frac{\sqrt{|T^2 -t^2|}}{2 \pi |n|} \Big)
\psi
\Big(\frac{\sqrt{|T^2 -t^2|}}{2 \pi |m|} \Big) \Big(\frac{|m|}{|n|}\Big)^{it}
\frac{dt}{2 \pi}.
\end{equation}
In the right hand side of \eqref{eq:F0integral}, write
\begin{equation}
\label{eq:F0integralsmoothedDecomposition}
\intR W(t) |F_0(it)|^2 \frac{dt}{2\pi} = M.T. + E(x,T),
\end{equation}
where $M.T.$ corresponds to the terms with $|m|=|n|$ and $E(x,T)$ are the terms with $|m| \neq |n|$.  We will bound $E(x,T)$ with absolute values.  We then have
\begin{equation}
\label{eq:MainTermDef}
 M.T. = \frac{|\rho^*(1)|^2}{\cosh( \pi T)} \sum_{n \neq 0} \frac{|\tau_{iT}(n)|^2}{|n|} J(n,n) (1 + e(2nx)),
\end{equation}
and
\begin{equation*}
 |E(x,T)| \ll |\zeta(1+2iT)|^{-2} \sum_{h \neq 0} \sum_n \frac{|\tau_{iT}(n)| |\tau_{iT}(n+h)|}{\sqrt{|n(n+h)|}} |J(n,n+h)|.
\end{equation*}

We will show that
\begin{equation}
\label{eq:MTasymptotic}
 M.T. = \frac{3}{\pi} \log(1/4 + T^2) \int_0^{\infty} \psi^2(y) \Big[1 + B_q(1,T) J_0\Big(\frac{\theta T}{y}\Big)\Big] \frac{dy}{y}
 + O(\eta^{1/2} \log T),
\end{equation}
and that
\begin{equation}
\label{eq:F0integralErrorTerm}
E(x,T) \ll \eta^{-3/2} (\log T)^{8/9 + \varepsilon}.
\end{equation} 
At this point we can also explain how to remove the smoothing factor $W$.  We have
\begin{align}
\begin{split}
\label{eq:F0unsmoothedUB}
 \int_{|t| \leq T- \eta T} |F_0(it)|^2 dt \leq \intR W_{+}(t) |F_0(it)|^2 dt 
 \\
 \int_{|t| \leq T- \eta T} |F_0(it)|^2 dt \geq \intR W_{-}(t) |F_0(it)|^2  dt, 
 \end{split}
\end{align}
where $W_{+}$ and $W_{-}$ satisfy the same properties as $W$, but with different choices of constants $c_i$. 
Since the approximations of $M.T.$ and $E(x,T)$ are independent of the choice of $W$, we can claim the same asymptotic  for $\int_{|t| \leq T - \eta(T) T} |F_0(it)|^2 dt$, showing that Proposition \ref{prop:bulkrangeIntro} follows from \eqref{eq:MTasymptotic} and \eqref{eq:F0integralErrorTerm}.

Now we analyze $J(m,n)$.  
Note that $m$ and $n$ are supported on 
\begin{equation*}
\eta^{1/2} T  \ll |m|, |n|
\ll T.   
\end{equation*}
By Stirling's approximation (see \eqref{eq:gammaVStirling}), we have
that $J(m,n) = \intR j(m,n,t,T)
(|m|/|n|)^{it} dt$, where $j$ is a function satisfying
$\frac{\partial^k}{\partial t^k} j(m,n,t,T) \ll \eta^{-1/2} T^{-1}
(\eta T)^{-k} $. Thus by standard
integration by parts, if $|m| \neq |n|$,
\begin{equation}
\label{eq:Jbound}
 J(m,n) \ll_k   \frac{\eta^{-1/2}}{[\eta T
\log(|m|/|n|)]^k}, \qquad k=0,1,2,\dots.
\end{equation}

\subsection{The off-diagonal terms}
By 
\cite[Theorem 1.2]{Holowinsky}, we have
\begin{equation*}
 \sum_{n \leq x} |\tau_{iT}(n) \tau_{iT}(n+h)| \ll x (h \log x)^{\varepsilon}   P_{iT}(x)^2,
\end{equation*}
where
\begin{equation*}
 P_{iT}(x) = \frac{1}{\log x} \prod_{p \leq x} \Big(1 + \frac{|\tau_{iT}(p)|}{p} \Big).
\end{equation*}
We claim
\begin{equation}
\label{eq:EiTBound}
 P_{iT}(T) \ll (\log T)^{1/3} |\zeta(1+2iT)|^{7/9} |\zeta(1+4iT)|^{-1/9}.
\end{equation}
To see this, we
begin with the elementary inequality
\begin{equation}
\label{eq:polynomialinequality}
 |x| \leq \frac{1}{18} (8 + 11 |x|^2 - |x|^4),
\end{equation}
valid for any $x \in [-2,2]$.  
Applying this with $x= \tau_{iT}(p) = p^{iT} + p^{-iT}$, we deduce
\begin{equation*}
|\tau_{iT}(p)| \leq \frac{1}{18}\Big(24 + 7(p^{2iT} + p^{-2iT}) -(p^{4iT} + p^{-4iT}) \Big),
\end{equation*}
which implies in turn
\begin{equation*}
P_{iT}(x) \ll  (\log x)^{-1} \prod_{p \leq x} \Big(1 + \frac{1}{p}\Big)^{4/3} \prod_{p \leq x} \Big|1 + \frac{1}{p^{1+2iT}} \Big|^{7/9} \prod_{p \leq x}\Big|1 + \frac{1}{p^{1+4iT}} \Big|^{-1/9}.
\end{equation*}
Of course, $\prod_{p \leq x} (1+ p^{-1}) \ll \log x$.   
Finally, using \eqref{eq:logzetaapproximation} gives \eqref{eq:EiTBound}.

Using \eqref{eq:Jbound}, we conclude
\begin{equation}
\label{eq:offdiagonal}
|E(x,T)| 
\ll \frac{\eta^{-3/2} (\log T)^{2/3+\varepsilon}}{|\zeta(1+2iT)|^{\frac{4}{9}} |\zeta(1+4iT)|^{\frac{2}{9}}}.
\end{equation}

\begin{mylemma}
\label{lemma:zetalemma}
For $t \gg 1$, we have
\begin{equation}
\label{eq:zetalemma}
 |\zeta(1+it)|^{-2} |\zeta(1+2it)|^{-1} \ll (\log t)^{1+\varepsilon}.
\end{equation}
\end{mylemma}
Lemma \ref{lemma:zetalemma} 
completes the proof of \eqref{eq:F0integralErrorTerm}.

Note that a naive application of the Vinogradov--Korobov bound would only give $(\log t)^{2+\varepsilon}$ in \eqref{eq:zetalemma}, leading to a bound on $E(x,T)$ that is worse than trivial.  

\begin{proof}[Proof of Lemma \ref{lemma:zetalemma}]
 Using \eqref{eq:logzetaapproximation}, we have
\begin{equation*}
\log\Big( |\zeta(1+it)|^{-2} |\zeta(1+2it)|^{-1} \Big) 
=
\sum_{p \leq x} \frac{-2 \cos(t \log p) - \cos(2t \log p)}{p} + O(1),
\end{equation*}
where $\log x  \asymp (\log t)^{2/3+\varepsilon}$.  Note that $-2 \cos \theta - \cos(2 \theta) \leq \frac{3}{2}$ for all $\theta \in \mr$, since $\frac32 + 2 \cos \theta + \cos(2 \theta) = \frac12(1+2 \cos \theta)^2$.  Thus
\begin{equation*}
 \log\Big( |\zeta(1+it)|^{-2} |\zeta(1+2it)|^{-1} \Big)
 \leq \tfrac{3}{2} \log \log x +O(1).
\end{equation*}
Simplifying completes the proof.
\end{proof}
One may apply many other inequalities in place of \eqref{eq:polynomialinequality}, leading to alternative forms of \eqref{eq:offdiagonal}.  We made no attempt to optimize the error term in \eqref{eq:F0integralErrorTerm}.

\subsection{The diagonal terms}
\label{section:diagonaltermsGeneralx}
According to \eqref{eq:MainTermDef}, write $M.T. = MT_{0} + MT_{x}$, say.  It suffices to study $MT_{x}$, since $MT_{0}$ is the special case $x=0$.
By a rearrangement, we have
\begin{equation*}
MT_{x} =  \frac{|\rho^*(1)|^2}{\cosh(\pi T)}
\sum_{n=1}^{\infty}
  \frac{|\tau_{iT}(n)|^2}{n} (e(2nx) + e(-2nx)) J(n,n).
\end{equation*}
Using the Dirichlet approximation theorem, we have
\begin{equation*}
 MT_{x} = \frac{|\rho^*(1)|^2}{\cosh( \pi T)} \sum_{n =1}^{\infty} \frac{|\tau_{iT}(n)|^2}{n} J(n,n)
 \Big[
 e\Big(\frac{an}{q} \Big) e(\theta n) + e\Big(\frac{-an}{q} \Big) e(-\theta n)
 \Big],
\end{equation*}
where $(a,q) = 1$, $q \leq Q$, and $|\theta q| \leq Q^{-1}$.  Here $Q$ is a parameter at our disposal to be chosen later.

Converting to Dirichlet characters via Lemma \ref{lemma:changeofbasis}, we obtain
\begin{equation*}
 MT_{x} = \frac{|\rho^*(1)|^2}{\cosh( \pi T)} 
 \sum_{d|q} \frac{1}{\phi(q/d)} 
\sum_{\pm}
 \sum_{\chi \shortmod{q/d}} \tau(\overline{\chi}) \chi(\pm a)
 \sum_{n =1}^{\infty} \frac{|\tau_{iT}(dn)|^2 \chi(n)}{dn} J(dn,dn)
  e(\pm \theta d n).
\end{equation*}
Let $J_{\pm \theta}(x) = e(\pm \theta x) J(x,x)$.  It follows from the definition \eqref{eq:Jdef} that
\begin{equation*}
 \frac{d^j}{dx^j} J(x,x) \ll x^{-j}.
\end{equation*}

Next, define the Mellin transform
\begin{equation*}
 \widetilde{J}_{\pm \theta} (u) = \int_0^{\infty}  J_{\pm \theta}(x) x^u \frac{dx}{x}.
\end{equation*}
Note
$\frac{d^j}{dx^j} J_{\pm \theta}(x) \ll x^{-j} (1+ |\theta T|^j)$.
Therefore, repeated integration by parts shows 
\begin{equation*}
|\widetilde{J}_{\pm \theta}(\sigma + iv)| \ll_{\sigma} T^{\sigma} \Big(1 + \frac{|v|}{1+|\theta T|}\Big)^{-100},
\end{equation*}
which we will use for $|\theta T| \ll T^{\varepsilon}$.
If $|\theta T| \gg T^{\varepsilon}$, then a stationary phase argument (see \cite[Proposition 8.2]{BKY}) shows
\begin{equation*}
 |\widetilde{J}_{\pm \theta} (\sigma + iv)| \ll \frac{T^{\sigma}}{|\theta T|^{1/2}} \Big(1 + \frac{|v|}{|\theta T|}\Big)^{-100}.
\end{equation*}
We may combine the two cases with the single bound
\begin{equation*}
 |\widetilde{J}_{\pm \theta} (\sigma + iv)| \ll \frac{T^{\sigma+\varepsilon}}{1+|\theta T|^{1/2}} \Big(1 + \frac{|v|}{T^{\varepsilon} + |\theta T|}\Big)^{-100}.
\end{equation*}

By the Mellin inversion formula, we have
\begin{equation*}
 MT_{x} =
\frac{|\rho^*(1)|^2}{\cosh( \pi T)} 
 \sum_{d|q} \frac{1}{\phi(q/d)} 
\sum_{\pm}
 \sum_{\chi \shortmod{q/d}} \tau(\overline{\chi}) \chi(\pm a) 
  \frac{1}{2 \pi i} \int_{(\varepsilon)} \widetilde{J}_{\pm \theta}(s) d^{-1-s} \mathcal{L}_d(1+s) ds,
\end{equation*}
where
\begin{equation}
\label{eq:LdsDef}
 \mathcal{L}_d(s) = \sum_{n =1}^{\infty} \frac{|\tau_{iT}(dn)|^2 \chi(n)}{n^{s}}.
\end{equation}
By a variation on \eqref{eq:Zformula}, one may readily derive that
\begin{equation*}
\mathcal{L}_d(s) = L(s, \chi)^2 L(s+iT, \chi) L(s-iT, \chi) A_d(s),
\end{equation*}
where $A_d(s)$ is given by an absolutely convergent Euler product for $\text{Re}(s) > 1/2$, satisfying a bound $A_d(s) \ll_{\sigma} d^{\varepsilon}$, for $\text{Re}(s) \geq \sigma > 1/2$.  Moreover, $\mathcal{L}_d(s)$ is analytic in this same region except for poles at $s=1, 1 \pm iT$ only when $\chi$ is the principal character.

Now let $MT_{x}'$ denote the contribution to $MT_{x}$ from $\chi$ nonprincipal.  To estimate $MT_{x}'$, we shift the contour to the line $-1/2 + \varepsilon$.  The bound we obtain from these terms is
\begin{multline*}
MT_x' \ll T^{\varepsilon} \sum_{\pm}
 \sum_{d|q} \frac{1}{\phi(q/d)} 
 \sum_{\chi \shortmod{q/d}} \sqrt{q/d} 
 \\
 \intR |\widetilde{J}_{\pm \theta} (-1+\sigma+ iv)| 
  |L(\sigma + iv, \chi)^2 L(\sigma + iv+iT, \chi) L(\sigma + iv-iT, \chi)| dv.
\end{multline*}
By H\"older's inequality and Proposition \ref{prop:JM}, we deduce
\begin{equation}
\label{eq:CriticalLineBoundFromJutilaMotohashi}
MT_{x}' \ll  T^{-1/2+\varepsilon}
  q^{1/2} (T^{2/3} + |\theta T|)^{1/2} \ll (T^{-1/6} Q^{1/2}  + Q^{-1/2}) T^{\varepsilon}.
\end{equation}
Any choice of $Q$ with $T^{\delta} \ll Q \ll T^{1/3-\delta}$, for some $\delta > 0$, shows $MT_x' \ll T^{-\delta'}$, for some $\delta'>0$.

Finally, consider the contribution from $\chi = \chi_0$, denoted $MT_x^{0}$.  Using $\tau(\chi_0) = \mu(q/d)$, we see that
\begin{equation}
\label{eq:MTx0}
MT_x^{0} = \frac{|\rho^*(1)|^2}{\cosh( \pi T)} 
\sum_{\pm}
  \frac{1}{2 \pi i} \int_{(\varepsilon)} \widetilde{J}_{\pm \theta} (s) Z_q(1+s) ds,
\end{equation}
where
\begin{equation*}
Z_q(s) = \sum_{n=1}^{\infty} \frac{|\tau_{iT}(n)|^2}{n^{s}} \frac{\mu(q/(n,q))}{\varphi(q/(n,q))}.
\end{equation*}
By \eqref{eq:Zformula}, 
we derive
\begin{equation*}
 Z_q(s) 
 = \frac{\zeta^2(s) \zeta(s+iT) \zeta(s-iT)}{\zeta(2s)}
 B_q(s,T),
\end{equation*}
where
\begin{equation}
\label{eq:BqDef}
 B_q(s,T) = \prod_{p^{q_p} || q} \frac{\sum_{j=0}^{\infty} 
\frac{|\tau_{iT}(p^j)|^2}{p^{js}} \frac{\mu(p^{q_p}/(p^j,p^{q_p}))}{\varphi((p^{q_p}/(p^j,p^{q_p}))}
}{\sum_{j=0}^{\infty} 
\frac{|\tau_{iT}(p^j)|^2}{p^{js}} }.
\end{equation}

Now we shift the contour to $\sigma = -1/2 + \varepsilon$, picking up residues at $s=0$ as well as $\pm iT$.
The poles at $\pm iT$ give small residues, due to the decay of $\widetilde{J_{\pm \theta}}(s)$ at this height (since $|\theta T| \ll T^{1-\varepsilon}$,  which is implied by $Q \gg T^{\delta}$).  The contribution on the new contour is obviously no larger than the one obtained with \eqref{eq:CriticalLineBoundFromJutilaMotohashi}.  
This discussion now shows
\begin{equation*}
MT_x = \frac{|\rho^*(1)|^2}{\cosh( \pi T)} 
  \text{Res}_{s=0} \Big( \sum_{\pm} \widetilde{J_{\pm \theta}}(s) Z_q(1+s) \Big)
+  O(T^{-1/6} Q^{1/2}  + Q^{-1/2}) T^{\varepsilon}).
\end{equation*}

\subsection{Evaluation of the residue}
\label{section:residue}
First we examine the Dirichlet series aspect of the residue.  By the Vinogradov--Korobov bound (Lemma \ref{lemma:VinogradovKorobov}), we have
\begin{equation*}
\frac{\zeta^2(s) \zeta(s+iT) \zeta(s-iT)}{\zeta(2s)} = \frac{|\zeta(1+2iT)|^2}{\zeta(2)} \Big(\frac{1}{s^2} + \frac{O((\log T)^{2/3+\varepsilon})}{s} + \dots\Big).
\end{equation*}
Of course, $B_q(s,T) = B_q(1,T) + B_q'(1,T) s + O(s^2)$, and a calculation shows $B_q(1,T)$ and $B_q'(1,T)$ are $O(q^{-1+\varepsilon})$.
Finally, we claim that $\widetilde{J_{\theta}}(0) = O(\log \log T)$, and
\begin{equation}
\label{eq:JMellinResidueCalculation}
\sum_{\pm} \widetilde{J}_{\pm \theta}'(0) = \frac{\pi^2}{2}  \log T \int_0^{\infty} \psi^2(y) J_0\Big(\frac{\theta T}{y}\Big) \frac{dy}{y} + O(\eta^{1/2} \log T).
\end{equation}
Taking these claims on $\widetilde{J}_{\pm \theta}$ for granted, we may then derive
\begin{equation}
\label{eq:MTx}
MT_{x} = \frac{\pi^2}{2} \frac{|\rho^*(1)|^2}{\cosh(\pi T)} \frac{|\zeta(1+2iT)|^2}{\zeta(2)}  B_q(1,T)  \log T \int_0^{\infty} \psi^2(y) J_0\Big(\frac{\theta T}{y}\Big) \frac{dy}{y} + O(\eta^{1/2} \log T),
\end{equation}
where we have used $\eta^{1/2} \log T \gg (\log T)^{2/3+\varepsilon}$.  Evaluation of the constant shows this is consistent with \eqref{eq:MTasymptotic}.

Now we prove the claims on $J_{\theta}$.
We have
\begin{equation}
\label{eq:Jtheta'tilde0def}
 \widetilde{J}_{\theta}'(0) = \int_0^{\infty} e(\theta x) J(x,x) \log x \frac{dx}{x},
\end{equation}
and likewise $\widetilde{J_{\theta}}(0)$ is given by the same integral, but without the $\log x$ factor.  Trivially bounding the integral shows $J_{\theta}(0) = O(\log \log T)$, since $J(x)$ is supported on $\eta^{1/2} T \ll x \ll T$, and $\log(1/\eta) \ll \log \log T$.

Moreover, $\log x = \log T + O(\log \log T)$, on the support of $J(x)$.  Then inserting the definition of $J$ and applying some changes of variables, we obtain 
\begin{equation*}
  \widetilde{J}_{\theta}'(0) = \log T \int_0^{\infty} \psi^2(y) 
 \Big[ \intR W(t) |\gamma_{V_T}(1/2+it)|^2 \cosh( \pi T) e^{i \theta \frac{\sqrt{T^2-t^2}}{y}} \frac{dt}{2\pi} \Big]
  \frac{dy}{y},
\end{equation*}
plus an error of size $O(\log \log T)$.

Stirling's formula implies $\cosh(\pi T) |\gamma_{V_T}(1/2+it)|^2 = \frac{\pi^2}{2} (T^2-t^2)^{-1/2} ( 1+O(T^{-1+\varepsilon}))$, for $t$ in the support of $W$.
Hence
\begin{equation}
\label{eq:JthetaMellinPrimeAtZero}
\widetilde{J}_{\theta}'(0) = \frac{\pi^2}{2} \log T \int_0^{\infty} \psi^2(y) 
 \Big[ \intR W(t) (T^2-t^2)^{-1/2}  e^{i \theta \frac{\sqrt{T^2-t^2}}{y}} \frac{dt}{2 \pi} \Big] \frac{dy}{y} + O(\log \log T).
\end{equation}
It is not difficult to see that
\begin{equation*}
\intR W(t) (T^2-t^2)^{-1/2}  e^{i \theta \frac{\sqrt{T^2-t^2}}{y}} dt
= \int_{-T}^{T} (T^2-t^2)^{-1/2}  e^{i \theta \frac{\sqrt{T^2-t^2}}{y}} dt + O(\eta^{1/2} \log T).
\end{equation*}
By the change of variables $t = T \cos(\varphi)$ and \cite[8.411.1]{GR}, we have
\begin{equation*}
\int_{-T}^{T} (T^2-t^2)^{-1/2}  \cos\Big(\theta \frac{\sqrt{T^2-t^2}}{y}\Big) dt
= \pi J_0\Big(\frac{\theta T}{y}\Big).
\end{equation*}

Putting all the calculations together, we derive \eqref{eq:JMellinResidueCalculation}, which completes the proof of Proposition \ref{prop:bulkrangeIntro}.

\section{Large values of $\Delta$ via shifted divisor sum}
\label{section:spectralapproach}
Having extracted the main term in Section \ref{section:largeDelta}, we may make a small notational simplification by restricting attention to $t \geq 0$ only.  The same bound holds for $t<0$ by symmetry.

The goal of this section is to prove Proposition \ref{prop:biggisht}.
We begin with some simplifications of a general nature.
\begin{mylemma}
\label{lemma:F0integralSimplified}
 Suppose that $T^{\delta} \ll \Delta \leq \eta(T) T$, where $0 < \eta(T) < \delta$ for some small $\delta > 0$.  Then
 \begin{equation}
  \int_{\frac12 \Delta \leq T - t \leq \Delta} |F_0(it)|^2 dt \ll  \frac{|\zeta(1+2iT)|^{-2}}{\sqrt{\Delta T}} I(\Delta, T, x,  N) + T^{-100},
 \end{equation}
where
\begin{equation}
\label{eq:IDeltaTxNdef}
I(\Delta, T,x, N) = \intR w_1\Big(\frac{T-t}{\Delta}\Big) \Big|\sum_{n \geq 1}
\frac{\tau_{iT}(n)
e(nx)}{n^{1/2 + it}} w_2
\Big(\frac{n}{N} \Big)\Big|^2 dt, \qquad N = \sqrt{\Delta T},
\end{equation}
and $w_1, w_2$ are certain fixed smooth weight functions with compact support on the positive reals.
\end{mylemma}
\begin{proof}
The essential part of the proof is a separation of variables argument.  We will need variations on this argument in other parts of the paper, so we provide full details here.

Using Stirling's
approximation \eqref{eq:gammaVStirling}, and the inequality
$|\sum_{n \neq 0} a_n |^2 \leq 2|\sum_{n > 0} a_n |^2 + 2|\sum_{n > 0} a_{-n}
|^2$, we obtain
\begin{equation*}
 \int \limits_{\frac12 \Delta \leq T-t \leq \Delta} |F_0(it)|^2 dt
 \ll \frac{|\zeta(1+2iT)|^{-2}}{\sqrt{\Delta T}} 
 \int \limits_{\frac12 \Delta \leq T-t \leq \Delta} \Big|\sum_{n \geq 1}
\frac{\tau_{iT}(n)
e(nx)}{n^{1/2 + it}} \psi
\Big(\frac{\sqrt{T^2-t^2}}{2 \pi n} \Big)\Big|^2 dt.
\end{equation*}
The support on $\psi$ means that the $n$-sum is supported on $n \asymp N$, where
\begin{equation}
\label{eq:Ndef}
N = \sqrt{\Delta T}.
\end{equation}
This integral is not yet quite in the form $I(\Delta, T, x, N)$.
For this,
we multiply $\psi$ by a redundant factor $w_2(n/N)$, such that $w_2(n/N) = 1$ for all $n$ in the support of $\psi$, yet such that $w_2$ is smooth of compact support.  Then we use  $\psi(y) = \frac{1}{2\pi i} \int_{(0)} \widetilde{\psi}(-u) y^u du$, and apply the Cauchy--Schwarz inequality to the $u$-integral, obtaining
\begin{equation}
\label{eq:psiseparation1}
 \Big| \sum_n a_n \psi
\Big(\frac{\sqrt{T^2-t^2}}{2 \pi n} \Big)\Big|^2 \ll \intR |\widetilde{\psi}(iu) | \Big| \sum_n a_n n^{-iu} \Big|^2 du,
\end{equation}
using the fact that $\intR |\widetilde{\psi}(iv)| dv \ll 1$ (the reader may recall that $\psi$ was fixed at the beginning of the paper and is independent of all relevant parameters).

In our application, we may truncate the $u$-integral at $|u| \leq \frac{\Delta}{100}$, leading to an error that is at most $O(T^{-100})$, using the rapid decay of $\widetilde{\psi}$, and the fact that $\Delta \gg T^{\delta}$.  That is,
\begin{multline*}
 \int \limits_{\frac12 \Delta \leq T-t \leq \Delta} \Big|\sum_{n \geq 1}
\frac{\tau_{iT}(n)
e(nx)}{n^{1/2 + it}} \psi
\Big(\frac{\sqrt{T^2-t^2}}{2 \pi n} \Big)\Big|^2 dt 
\\
\ll 
\int \limits_{\frac12 \Delta \leq T-t \leq \Delta}
\thinspace
\int \limits_{|u| \leq \frac{\Delta}{100}} |\widetilde{\psi}(-iu)|
\Big|\sum_{n \geq 1}
\frac{\tau_{iT}(n)
e(nx)}{n^{1/2 + it+iu}} w_2\Big(\frac{n}{N}\Big)\Big|^2 dt du + O(T^{-100}).
\end{multline*}
Next we change variables $t \rightarrow t-u$ and over-extend the $t$-range of integration to $\frac{49 }{100} \Delta \leq |t-T| \leq \frac{101 }{100} \Delta$, which is valid due to the positivity of the integrand.  In this way the variables $u$ and $t$ are separated, and the $u$-integral can be bounded again using $\intR |\widetilde{\psi}(iv)| dv \ll 1$.

Finally, by attaching a smooth weight function $w_1(\frac{T-t}{\Delta})$ supported on the positive reals, we obtain the desired result.
\end{proof}

\begin{proof}[Proof of Proposition \ref{prop:biggisht}]
The basic approach is similar to that of Proposition \ref{prop:bulkrangeIntro}, in that
we will solve a shifted divisor problem.  
In light of Lemma \ref{lemma:F0integralSimplified}, we need to show
\begin{equation}
\label{eq:biggisht2}
 I(\Delta, T, x, N) \ll |\zeta(1+2iT)|^2(\Delta  \log T + N T^{-\delta'} + N S_{\Delta}),
\end{equation}
for $\Delta \gg T^{25/27+\varepsilon}$, and some $\delta'>0$.  
Opening the square and integrating in $t$, we obtain
\begin{equation}
\label{eq:IDeltaTxNFormula}
 I(\Delta, T, x, N) = \sum_{m,n \geq 1} \tau_{iT}(m) \tau_{iT}(n)
e((m-n)x) \frac{ w_2\big(\frac{n}{N}\big) w_2\big(\frac{m}{N}\big)}{\sqrt{mn}} \intR w_1\Big(\frac{T-t}{\Delta}\Big) \Big(\frac{n}{m}\Big)^{it} dt.
\end{equation}
By Lemma \ref{lemma:diagonaltermsum}, the terms with $m=n$
contribute to $I(\Delta, T, x, N)$ an amount
\begin{equation*}
 \ll \Delta |\zeta(1+2iT)|^2 \log T,
\end{equation*}
which leads to the first bound in \eqref{eq:biggisht2}.

Our main job is to bound the terms with $m\neq n$, in which case we set $m = n +h$.  By simple
estimations, these terms contribute to $I(\Delta, T,x,N)$ an
amount 
\begin{equation*}
 \frac{\Delta}{N} \sum_{h \neq 0} e(hx) \sum_{n} \tau_{iT}(n) \tau_{iT}(n+h) K(n,h),
\end{equation*}
where 
\begin{equation*}
K(n,h) = 
\frac{w_2(\frac{n}{N})}{(\frac{n}{N})^{1/2}} 
\frac{w_2(\frac{n+h}{N})}{(\frac{n+h}{N})^{1/2}}
\Big(\frac{n+h}{n}\Big)^{-iT} 
\widehat{w_1}\Big(-\frac{\Delta}{2\pi} \log\Big(\frac{n+h}{n}\Big)\Big).
\end{equation*}
In particular, $K(n,h)$ is
supported on $n \asymp N$ and satisfies
\begin{equation*}
 \frac{\partial^{j}}{\partial u^j} K(u,v) \ll R^j N^{-j} 
\Big(1 + \frac{\Delta |v|}{N}\Big)^{-A}, \qquad j =0,1,2,\dots,
\end{equation*}
with $A  > 0$ arbitrary, and where
\begin{equation}
 \label{eq:Rdef}
 R = \Delta^{-1} T.
\end{equation}
Proposition \ref{prop:SCS} gives
\begin{equation}
\label{eq:YoungThm8.1basicapplication}
 \frac{\Delta}{N} \sum_{h \neq 0} e(hx) \Big(\sum_n \tau_{iT}(n) \tau_{iT}(n+h) K(n,h) 
 - M.T. \Big) \ll T^{\frac13+\varepsilon} N^{\frac12} R^2 +
T^{\frac16+\varepsilon} N^{\frac34} R^{\frac12},
\end{equation}
where $M.T.$ is a main term (off-diagonal) to be discussed presently.
A short calculation shows that the two error terms in \eqref{eq:YoungThm8.1basicapplication} are
$O((\Delta T)^{1/2} T^{-\delta})$ for $\Delta \gg T^{\frac{25}{27} +
\varepsilon}$.  It turns out that the main terms are surprisingly difficult to estimate with uniformity in $x$, and we next turn to this problem.

The off-diagonal main term is defined by
\begin{equation*}
M.T. = \sum_{\pm} \frac{|\zeta(1+2iT)|^2}{\zeta(2)} \sigma_{-1}(h) \int_{\max(0,-h)}^{\infty} 
(y+h)^{\mp iT} y^{\pm iT} K(y,h) dy.
\end{equation*}
These two terms then give to $I(\Delta, T, x, N)$ an amount
\begin{multline*}
\frac{|\zeta(1+2iT)|^2}{\zeta(2)} \sum_{\pm} \frac{\Delta}{N} \sum_{h \neq 0} e(hx) \sigma_{-1}(h) 
\\
\int_{0}^{\infty} 
 \frac{w_2(\frac{y}{N})}{(\frac{y}{N})^{1/2}} 
\frac{w_2(\frac{y+h}{N})}{(\frac{y+h}{N})^{1/2}}
\Big(\frac{y+h}{y}\Big)^{-iT\mp iT}
\widehat{w_1}\Big(-\frac{\Delta}{2\pi} \log\Big(\frac{y+h}{y}\Big)\Big) dy.
\end{multline*} 
When the choice of sign makes the exponent of $(y+h)/y$ be $-2iT$, then a simple integration by parts argument shows that
\begin{equation*}
\int_{0}^{\infty} 
  \frac{w_2(\frac{y}{N})}{(\frac{y}{N})^{1/2}} 
\frac{w_2(\frac{y+h}{N})}{(\frac{y+h}{N})^{1/2}}
\Big(\frac{y+h}{y}\Big)^{-2 iT} 
\widehat{w_1}\Big(-\frac{\Delta}{2\pi} \log\Big(\frac{y+h}{y}\Big)\Big) dy \ll N \Big(1 + \frac{|h|T}{N}\Big)^{-10}.
\end{equation*}
Hence this off-diagonal term contributes to $I(\Delta, T, x, N)$ an amount that is
\begin{equation*}
\ll |\zeta(1+2iT)|^2 \frac{\Delta}{N} \sum_{h \neq 0} N \Big(1 + \frac{|h|T}{N}\Big)^{-10} \ll |\zeta(1+2iT)|^2 \Delta \frac{N}{T} \ll |\zeta(1+2iT)|^2 \Delta,
\end{equation*}
since $N \ll T$, which is more than satisfactory for \eqref{eq:biggisht2}.  The other main term is more subtle, and gives
\begin{equation}
\label{eq:MTODdef}
MT_{OD} := \frac{|\zeta(1+2iT)|^2}{\zeta(2)}  \frac{\Delta}{N} \sum_{h \neq 0} e(hx) \sigma_{-1}(h) 
\int_{0}^{\infty} 
 \frac{w_2(\frac{y}{N})}{(\frac{y}{N})^{1/2}} 
\frac{w_2(\frac{y+h}{N})}{(\frac{y+h}{N})^{1/2}}
\widehat{w_1}\Big(-\frac{\Delta}{2\pi} \log\Big(\frac{y+h}{y}\Big)\Big) dy.
\end{equation}
We will show here
\begin{equation}
\label{eq:MTODbound}
 MT_{OD} \ll |\zeta(1+2iT)|^2 (\Delta \log T + N S_{\Delta}),
\end{equation}
where $S_{\Delta}$ is a function satisfying \eqref{eq:SDeltaProperty}.
It is quite easy to show with a trivial bound that $MT_{OD} \ll |\zeta(1+2iT)|^2 N$, so \eqref{eq:MTODbound} amounts to a logarithmic savings with the sum over $\Delta$ in dyadic segments (there are of course $O(\log T)$ such values of $\Delta$ over which to sum).  As a first simple approximation step, by taking Taylor expansions we have that
\begin{multline*}
 \int_{0}^{\infty} 
 \frac{w_2(\frac{y}{N})}{(\frac{y}{N})^{1/2}} 
\frac{w_2(\frac{y+h}{N})}{(\frac{y+h}{N})^{1/2}}
\widehat{w_1}\Big(-\frac{\Delta}{2\pi} \log\Big(\frac{y+h}{y}\Big)\Big) dy
\\
= \int_0^{\infty} \frac{w_2(\frac{y}{N})^2}{(\frac{y}{N})} 
\widehat{w_1}\Big(- \frac{\Delta h}{2 \pi y}\Big) dy + O\Big(|h| \Big(1 + \frac{|h|\Delta}{N}\Big)^{-100}\Big),
\end{multline*}
and hence by a trivial estimate on the $h$-sum, we have
\begin{multline}
\label{eq:MTODmainterm}
 MT_{OD} = \frac{|\zeta(1+2iT)|^2}{\zeta(2)} \Delta \sum_{h \neq 0} e(hx) \sigma_{-1}(h) 
\int_0^{\infty} \frac{w_2(y)^2}{y}  
\widehat{w_1}\Big(-\frac{\Delta h}{2\pi N y}\Big) dy 
\\
+ O\Big(\frac{N}{\Delta}|\zeta(1+2iT)|^2 \log T\Big).
\end{multline}
This error term is certainly $O(NT^{-\delta})$, since $\Delta \gg T^{25/27+\varepsilon}$.
By changing variables, we have 
\begin{equation*}
 MT_{OD} \ll |\zeta(1+2iT)|^2 N \int_0^{\infty} w_2
 \Big(\frac{y}{2\pi}\Big)^2 
 \Big| \frac{\Delta}{Ny}
 \sum_{h \neq 0} \sigma_{-1}(h) e(hx) \widehat{w_1}(-\frac{\Delta h}{N y})
 \Big| dy
+ O(NT^{-\delta}).
 \end{equation*}
 We need to control the sum over $h$, on average over $\Delta$ running over dyadic segments.  We pause our estimations of $MT_{OD}$ to state the following. 
 \begin{mylemma}
 \label{lemma:QxH}
 Let $w_1$ be a smooth function supported on the positive reals, and 
for $x \in \mr$ and $H \gg 1$, define
\begin{equation}
 R(x,H) = H^{-1} \sum_{h \neq 0} \sigma_{-1}(h) e(hx) \widehat{w_1}(h/H).
\end{equation}
Then
\begin{equation}
\label{eq:QsummedDyadically}
 \sum_{\substack{1 \ll H < \infty \\ H \text{dyadic}}} |R(x,H)| \ll 1,
\end{equation}
where the implied constant depends only on $w_1$ and is uniform in $x$.
\end{mylemma}
Remark.  The assumption that $w_1$ is supported on the positive reals, so in particular $w_1(0) = 0$,  is used crucially in the proof.  More generally, if $w_1$ is some fixed Schwartz-class function, then it is not hard to see that $R(x,H) \ll_{w_1} 1$.

Lemma \ref{lemma:QxH} suffices to show \eqref{eq:MTODbound}, as $H = y \frac{N}{\Delta} = y \frac{T^{1/2}}{\Delta^{1/2}}$, and $\Delta$ runs over dyadic segments.  This means that $H$ can be partitioned into two classes of dyadic segments (taking even powers and odd powers separately, for example).

Having accounted for all the terms, the proof of  Proposition \ref{prop:biggisht} is complete.
\end{proof}

\begin{proof}[Proof of Lemma \ref{lemma:QxH}]
We open $\sigma_{-1}(h)$ as a Dirichlet convolution, obtaining
\begin{equation*}
R(x,H) =  
H^{-1}\sum_{a \geq 1} a^{-1} 
\sum_{b \neq 0} e(abx) 
\widehat{w_1}\Big(\frac{ab}{H}\Big).
\end{equation*}
For larger values of $a$, say $a > K \geq 1$, we estimate trivially, obtaining
\begin{equation*}
H^{-1} \sum_{a > K} a^{-1} \sum_{b \neq 0} e(abx) 
\widehat{w_1}\Big(\frac{ab}{H}\Big) \ll \frac{1}{K}.
\end{equation*}
For $a \leq K$, we apply Poisson summation, getting
\begin{equation*}
H^{-1} \sum_{a \leq K} a^{-1} \sum_{b \neq 0} e(abx) 
\widehat{w_1}\Big(\frac{ab}{H}\Big) = \sum_{a \leq K} a^{-1} \Big( \frac{-\widehat{w_1}(0)}{H} + 
\frac{1}{a} \sum_{\nu \in \mz} w_1\Big(H\big(x-\frac{\nu}{a}\big)\Big) 
\Big).
\end{equation*}
Thus
\begin{equation*}
 R(x,H) =  \sum_{a \leq K} a^{-2} \sum_{\nu \in \mz} w_1\Big(H\big(x-\frac{\nu}{a}\big)\Big) + O\Big(K^{-1} + \frac{\log (eK)}{H} 
\Big).
\end{equation*}
Suppose that $w_1$ is supported on $[\alpha, \beta]$ where $0 < \alpha < \beta$.
If we assume that $K \leq \varepsilon H$ for $\varepsilon^{-1} > 2\beta$, then the sum over $\nu$ will capture at most one value of $\nu$, which must be the integer nearest $ax$.  
More precisely, we have that the sum over $\nu$ vanishes unless $\alpha \leq \|ax\| \frac{H}{a} \leq \beta$, where $\|ax \|$ denotes the distance from $ax$ to the nearest integer.  For each $a$, there are at most $1 + \log_2(\beta/\alpha)$ dyadic values of $H$ such that the sum over $\nu$ does not vanish.  Thus
\begin{equation}
\label{eq:w1nuHdyadicsumbound}
\sum_{\substack{H \gg \varepsilon^{-1} \\ H \text{ dyadic}}} \sum_{\nu \in \mz} \Big| w_1\Big(H\big(x-\frac{\nu}{a}\big)\Big) \Big| \ll 1. 
\end{equation}
We set $K =  \varepsilon H$, and by reversal of order of summation combined with \eqref{eq:w1nuHdyadicsumbound}, we derive
\begin{equation*}
\sum_{\substack{H \gg \varepsilon^{-1} \\ \text{dyadic} }} \sum_{a \leq \varepsilon H} a^{-2} \sum_{\nu \in \mz} \Big| w_1\Big(H\big(x-\frac{\nu}{a}\big)\Big) \Big|
\ll \sum_{a=1}^{\infty} a^{-2} \ll 1.
\end{equation*}
For $H \ll \varepsilon^{-1}$ we use only the trivial bound $Q(x, H) \ll 1$, which needs to be used $O_{\varepsilon}(1)$ times.
Thus we obtain that \eqref{eq:QsummedDyadically} holds.
\end{proof}

\section{The medium sizes of $\Delta$}
\label{section:smallmediumDelta}

\subsection{Statement of results}
The goal of this section is to collect some results which together prove Proposition \ref{prop:FboundmediumDelta}.  
Recall $I(\Delta, T, x, N)$ is defined by \eqref{eq:IDeltaTxNdef}.  One of the main results of this section is
\begin{myprop}
\label{prop:smallt}
 Suppose $T^{\delta} \ll \Delta \ll T^{1-\delta}$, for some $\delta > 0$.  Then 
\begin{equation}
I(\Delta, T, x, N) \ll N |\zeta(1+2iT)|^2 ((\log T)^{-50} +  S_{\Delta}),
\end{equation} 
where $S_{\Delta}$ has the properties described in Proposition \ref{prop:biggisht}.
\end{myprop}
Proposition \ref{prop:smallt} then implies Proposition \ref{prop:FboundmediumDelta}, via Lemma \ref{lemma:F0integralSimplified}.

The methods in this section rely on opening the divisor function $\tau_{iT}(n) = \sum_{ab=n} (a/b)^{iT}$ 
and do not rely on the spectral theory of automorphic forms but rather abelian harmonic analysis.  
The relative sizes of $a$ and $b$ play a major role in our estimations, and so we wish to localize $a$ and $b$ at this early stage.  To this end, define
\begin{equation}
\label{eq:IABdef}
I_{A,B}  = \intR w_1\Big(\frac{t}{\Delta}\Big) \Big|\sum_{a, b}
\frac{
e(abx)}{a^{1/2 -it} b^{1/2 + 2iT-it}}  f\Big(\frac{a}{A}\Big)  g\Big(\frac{b}{B}\Big) \Big|^2 dt,
\end{equation}
where $A, B \gg 1$, and $w_1,f,g$ are fixed smooth functions supported on the positive reals.  
For small values of $b$, it is better to not use a dyadic partition.  Instead,
let $B_0 \asymp T^{\alpha}$ for some small $0 < \alpha < 1/4$, and let $g_0$ be a smooth function such that $g_0(y) = 1$ for $y < 1$, and $g_0(y) = 0$ for $y > 2$.  Then set
\begin{equation}
\label{eq:IB0def}
 I_{B_0} = \intR w_1\Big(\frac{t}{\Delta}\Big) \Big| \sum_{a, b} \frac{e(ab x) }{a^{1/2-it} b^{1/2 + 2iT -it}} g_0\Big(\frac{b}{B_0}\Big) w_2\Big(\frac{ab}{N}\Big) \Big|^2 dt.
\end{equation}

We elucidate a connection between $I(\Delta, T, x, N)$, $I_{A,B}$, and $I_{B_0}$, with the following
\begin{mylemma}
\label{lemma:IDeltaDyadicDecomposition}
With certain $w, f, g, g_0, B_0$ as above, and with $AB \asymp N = \sqrt{\Delta T}$, we have
\begin{equation}
\label{eq:IDeltaB0ABdecomposition}
 I(\Delta, T, x, N) \ll I_{B_0} 
 + \log T \Big(\sum_{\substack{B_0 \ll B \ll N \\ \text{dyadic}}} 
 I_{A,B} \Big) + \frac{N \log T}{\Delta^{100}}.
\end{equation}
\end{mylemma}

\begin{proof}
In the definition of $I(\Delta, T, x, N)$ (that is, \eqref{eq:IDeltaTxNdef}), write $\tau_{iT}(n) = \sum_{ab=n} (a/b)^{iT}$, and then apply a partition of unity to the $b$-sum.  In this way, we have
\begin{equation*}
 I(\Delta, T, x, N) =  \intR w_1\Big(\frac{t}{\Delta}\Big) \Big|\sum_{j} \sum_{a, b} \frac{e(ab x) g_j(b)}{a^{1/2-it} b^{1/2 + 2iT -it}} w_2\Big(\frac{ab}{N}\Big) \Big|^2 dt,
\end{equation*}
where $g_j$ are elements of the partition.  
 
 First we choose $g_0(b/B_0)$ to be one element of the partition of unity.  Next we simply apply Cauchy's inequality to separate this part of the partition from all the other terms.  This explains the presence of $I_{B_0}$ in \eqref{eq:IDeltaB0ABdecomposition}; the main point here is we did not lose the factor $\log T$.
 
 We may assume that $g_0$ is such that $1-g_0(y/B_0)$ equals a dyadic partition of unity (meaning, each $g_j(b) = g(b/B)$, say with $b \asymp B$, where $AB = N$, and $B$ runs over a dyadic sequence), since we could start with a dyadic partition of unity of this type, and then write $g_0$ as a sum of elements of this dyadic partition.  Having located $b \asymp B$ with $g(b/B)$, $a$ is automatically localized by $a \asymp A$ by the support of $w_2$.  We are then free to multiply by $f(a/A)$ where $f$ is smooth of compact support, and such that $f(a/A) = 1$ for all $a$ in the support of $g(b/B) w_2(ab/N)$.
 In all, this shows that we may write
 \begin{equation*}
  I(\Delta, T, x, N) \ll I_{B_0} + \intR  w_1\Big(\frac{t}{\Delta}\Big) \Big|\sum_{\substack{B_0 \ll B \ll N \\ B \text{ dyadic}}} \sum_{a, b} \frac{e(ab x) f(a/A) g(b/B)}{a^{1/2-it} b^{1/2 + 2iT -it}} w_2\Big(\frac{ab}{N}\Big) \Big|^2 dt.
 \end{equation*}
Applying Cauchy's inequality shows an estimate \emph{almost} of the form \eqref{eq:IDeltaB0ABdecomposition}, the only difference being that instead of $I_{A,B}$, we have an expression of the form
\begin{equation*}
 \intR w_1\Big(\frac{t}{\Delta}\Big) \Big| \sum_{a, b} \frac{e(ab x) f(a/A) g(b/B)}{a^{1/2-it} b^{1/2 + 2iT -it}} w_2\Big(\frac{ab}{N}\Big) \Big|^2 dt.
\end{equation*}
That is, we have a weight function of the form $f(a/A) g(b/B) w_2(ab/N)$, and we would like to omit $w_2$ in order to separate the variables further.  Using a method very similar to that in the proof of Lemma \ref{lemma:F0integralSimplified}, we can do this, at the cost of an error term of size $\Delta^{-100} N \log T$.
\end{proof}

In the application to $I(\Delta, T, x, N)$ with $\Delta \gg T^{\varepsilon}$, we have $AB = N = (\Delta T)^{1/2}$.  However, for some technical reasons, when $\Delta \ll T^{\varepsilon}$, we need a minor generalization of the range of $A,B$.  For this reason, in the rest of this section we let
\begin{equation}
\label{eq:Msize}
 AB = M, \quad \text{where} \quad  T^{1/2-\varepsilon} \ll M \ll T^{1-\varepsilon}.
\end{equation}

For ease of reference, we collect a few lemmas together, deferring the proofs until later in this section.

\begin{mylemma}[Poisson in $b$]
\label{lemma:bigB}
Suppose $1 \ll \Delta \ll T^{1-\delta}$, and $T^{1/2} \ll B \ll T^{1-\delta}$, for some $\delta > 0$.  Then
\begin{equation}
I_{A,B} \ll \Big(\Delta + \frac{MT}{B^2}\Big) \log T.
\end{equation}
\end{mylemma}

\begin{mylemma}[Poisson in $a$]
\label{lemma:bigA}
Suppose $(\log T)^{100} \ll \Delta \ll T^{1-\delta}$ and $B \gg T^{\delta}$ for some $\delta > 0$.  Then 
\begin{equation}
I_{A,B} \ll M |\zeta(1+ 2iT)|^2 (\log T)^{-50}  + \Delta \frac{B^2}{M}  \log T + \Delta \log  T.
\end{equation}
\end{mylemma}
\begin{mylemma}[Determinant equation, Poisson in $b_2$]
\label{lemma:ShiftedConvolutionWithBsummed}
 Suppose that $T^{\delta} \ll \Delta \ll T^{1-\delta}$ and $T^{1/2-\delta} \ll M \ll T^{1-\delta}$ for some $\delta > 0$.  Then
 \begin{equation}
 \label{eq:IABboundFromShiftedConvolutionWithBsummed}
  I_{A,B} \ll \Delta \log T + \frac{M}{\Delta} +  T^{\varepsilon} \frac{A^2}{\Delta}  \frac{(\Delta T)^{1/4}}{M^{1/2}} (\Delta T)^{1/2} \Big(\frac{\Delta}{T}\Big)^{1/12}.
 \end{equation}
\end{mylemma}

Finally we require a result valid for $B \ll T^{\varepsilon}$.  
\begin{mylemma}[Determinant equation, Poisson in $a_2$]
\label{lemma:smallB}
Suppose that $B_0 = T^{\alpha}$, for some fixed $0 < \alpha < 1/4$, and $(\log T)^{100} \ll \Delta \ll T^{1-\delta}$, for some $\delta > 0$.  Then with $S_{\Delta}$ satisfying the same properties as in Proposition \ref{prop:biggisht}, we have
\begin{equation}
\label{eq:IB0bound}
 I_{B_0} \ll M |\zeta(1+ 2iT)|^2 ((\log T)^{-50} +  S_{\Delta} ) + \Delta \log T.
\end{equation}
\end{mylemma}
Remark.  For later discussions in Section \ref{section:verysmallDelta}, it is helpful to describe how the conclusion of Lemma \ref{lemma:smallB} is affected when we only assume that $w_1$ has support on $\mathbb{R}$ instead of the positive reals.  In this case, a bound of the form \eqref{eq:IB0bound} still holds, but we may only claim that $S_{\Delta} \ll 1$, instead of the stronger condition \eqref{eq:SDeltaProperty}.


Before proceeding to the proofs of these lemmas, we describe how they combine to prove Proposition \ref{prop:smallt}.  We need to bound all the terms on the right hand side of \eqref{eq:IDeltaB0ABdecomposition}.  Lemma \ref{lemma:smallB} treats $I_{B_0}$ with an acceptable bound.  Next we analyze $I_{A,B}$.  In the case $ T^{\alpha} \ll B \ll T^{1/2-\varepsilon}$, we apply Lemma \ref{lemma:bigA}.  In the case $B \gg T^{1/2+\varepsilon}$, i.e., $A \ll \Delta^{1/2} T^{-\varepsilon}$, we may apply Lemma \ref{lemma:bigB}.  Finally, in the case $T^{1/2-\varepsilon} \ll B \ll T^{1/2+\varepsilon}$, i.e., $\Delta^{1/2} T^{-\varepsilon} \ll A \ll \Delta^{1/2} T^{\varepsilon}$, then we apply Lemma \ref{lemma:ShiftedConvolutionWithBsummed}.
We have therefore covered all the ranges of $B$, showing Proposition \ref{prop:smallt}.

\subsection{Proof of Lemma \ref{lemma:bigB}}
The basic idea is that Poisson summation in $b$ inside the absolute square gives a savings by reducing the length of summation (provided $B$ is slightly larger than $T^{1/2}$).  Furthermore, this process is compatible  with the savings from the $t$-integration.

We have
\begin{equation*}
 \sum_{b}
\frac{e(abx)}{ b^{1/2+2iT-it}} g\Big(\frac{b}{B}\Big)  =  \sum_{l \in \mz} \intR e(axu + l u - \frac{(2T-t)}{2\pi}
\log{u})  u^{-1/2} g(u/B) du.
\end{equation*}
Write the phase here as $\phi(u) = (ax+l)u - \frac{2T-t}{2\pi} \log u$, which satisfies
\begin{equation*}
\phi'(u) = (ax+l) - \frac{2T-t}{2\pi u}.
\end{equation*}
We first claim that the integral is small unless there is a stationary point.  The second term in the derivative of the phase is $\asymp \frac{T}{B}$, and repeated integration by parts (see \cite[Lemma 8.1]{BKY}) shows that the integral is small (say, $O(T^{-100})$) unless $(ax+l) \asymp \frac{T}{B}$.  By periodicity, we may as well assume $|x| \leq 1/2$, in which case $a|x| \leq a \ll \frac{M}{B} = o(T/B)$ (recall \eqref{eq:Msize}), and therefore this means $l \asymp \frac{T}{B}$.  We think of $l+ax$ as a perturbation of $l$.  The stationary point occurs at $2 \pi u_0 = \frac{2T-t}{l+ax}$.  By \cite[Lemma 5.5.6]{Huxley} or \cite[Proposition 8.2]{BKY} (taking only the first term in the asymptotic expansion), we get an asymptotic formula 
in the form
\begin{equation*}
 \sum_{b}
\frac{e(abx)}{ b^{1/2+2iT-it}} g\Big(\frac{b}{B}\Big)  =  \sum_{l}
\frac{ \Big(\frac{2T-t}{2\pi e}\Big)^{-i(2T-t)}}{(l+ax)^{1/2-2iT+it}}
P\Big(\frac{2T-t}{B(l+ax)}\Big) + O((BT)^{-1/2}),
\end{equation*}
where $P$ is a function supported on $v \asymp 1$, satisfying $P^{(j)}(v) \ll_j 1$.  In fact, $P$ is a constant multiple of $g$.
The phase here takes the form
\begin{equation*}
 \frac{e^{i\varphi(t,T)}}{(l+ax)^{it - 2iT}},
\end{equation*}
for some function $\varphi$ independent of $a$ and  $l$.  
Altogether, this shows
\begin{equation*}
I_{A,B} \ll \intR w_1\Big(\frac{t}{\Delta}\Big) \Big| \sum_{a, l}
\frac{P(\frac{2T-t}{ B(l+ax)})  f\big(\frac{a}{A}\big)}{a^{1/2 -it} (l+ax)^{1/2 - 2iT+it}}
\Big|^2 dt + O\Big(\frac{\Delta A}{BT}\Big).
\end{equation*}
Note that this error term is $\frac{\Delta A}{BT}  = \frac{\Delta M}{B^2 T}$, which is acceptable.

Since $\frac{|t|}{T} \ll \frac{\Delta}{T} \ll T^{-\delta}$, and $\frac{a|x|}{l} \ll \frac{A}{T/B} = \frac{M}{T} \ll T^{-\delta}$, we can take Taylor expansions to obtain
\begin{equation*}
 P\Big(\frac{2T-t}{ B(l+ax)}\Big) = P\Big(\frac{2T}{Bl}\Big) + \sum_{j,k} P_{j,k}\Big(\frac{2T}{Bl}\Big) \Big(\frac{t}{T}\Big)^j \Big(\frac{ax}{l}\Big)^k + O(T^{-100}),
\end{equation*}
where $P_{j,k}$ is a finite linear combination of derivatives of $P$, with $j+k \geq 1$, and the total number of terms in the sum bounded by a function of $\delta$.  By Cauchy's inequality, we then have
\begin{equation*}
 I_{A,B} \ll \intR w_1\Big(\frac{t}{\Delta}\Big) \Big| \sum_{a, l}
\frac{P(\frac{2T}{ Bl})  f\big(\frac{a}{A}\big)}{a^{1/2 -it} (l+ax)^{1/2 - 2iT+it}}
\Big|^2 dt + O\Big(\frac{\Delta A}{BT}\Big),
\end{equation*}
where by abuse of notation, $P$ is meant to represent $P$ or one of the $P_{j,k}$ (of course the case of $P$ is the hardest for us to estimate, as the terms with $j+k \geq 1$ will be smaller than this by a factor $\ll T^{-\varepsilon}$, so the notation is not that abusive after all).

Change variables $a \mapsto da$ and $l \mapsto dl$ where now $(a,l) = 1$.  Define the coefficients
\begin{equation*}
 \gamma_{a,l} = \sum_{d=1}^{\infty} \frac{P(\frac{2T}{ Bdl})  f\big(\frac{ad}{A}\big)}{d^{1-2iT}},
\end{equation*}
so that
\begin{equation}
\label{eq:IABafterPoissonBandStationaryPhase}
 I_{A,B} \ll \intR w_1\Big(\frac{t}{\Delta}\Big) \Big| \sum_{(a, l)=1}
\frac{\gamma_{a,l}}{a^{1/2 -it} (l+ax)^{1/2 - 2iT+it}}
\Big|^2 dt + O\Big(\frac{\Delta A}{BT}\Big).
\end{equation}
Note that $\gamma_{a,l} = 0$ unless $a \asymp \frac{M l}{T}$.  We have the trivial bound $|\gamma_{a,l}| \ll 1$, and so $\sum_{a,l} \frac{|\gamma_{a,l}|^2}{al} \ll \log T$. 
To finish the proof, we will apply \eqref{eq:MontgomeryVaughanVersion2} to \eqref{eq:IABafterPoissonBandStationaryPhase}.  Here $r$ corresponds to the pair $(a,l)$, and $\lambda_r = \lambda_{a,l} = \log(\frac{l + ax}{a})$.  
Here we have
\begin{equation*}
 |\lambda_{a_1, l_1} - \lambda_{a_2,l_2}| = \Big|\log\Big(1 + \frac{a_2(l_1 + a_1 x) - a_1(l_2 + a_2 x)}{a_1 (l_2 + a_2 x)}\Big) \Big| \gg \frac{|a_2 l_1 - a_1 l_2|}{a_1 l_2},
\end{equation*}
and so $|\lambda_{a_1, l_1} - \lambda_{a_2,l_2}| \gg \frac{1}{a_1 l_2} \gg \frac{B}{A T}$ under the assumption $(a_1, l_1) \neq (a_2, l_2)$.  Therefore,
\begin{equation}
\label{eq:IABfinalBoundBPoisson}
 I_{A,B} \ll \frac{\Delta A}{BT}+ \sum_{(a,l) = 1} \frac{|\gamma_{a,l}|^2}{a l} \Big(\Delta + \frac{AT}{B}\Big)   \ll \Big(\Delta + \frac{MT}{B^2}\Big) \log T.
\end{equation}

\subsection{Proof of Lemma \ref{lemma:bigA}}

The initial steps of the proof are quite similar to that of Lemma \ref{lemma:bigB}, except that
we apply Poisson summation in $a$ instead of $b$.  The later steps are different because we encounter a more difficult counting argument in this problem.
We have
\begin{equation*}
\sum_{a=1}^{\infty} \frac{e(abx)}{a^{1/2-it}} f(a/A) = \sum_{q \in \mz} \intR e((bx - q)u +
\frac{t}{2\pi} \log u) u^{-1/2} f(u/A) du.
\end{equation*}
By \cite[Lemma 8.1]{BKY}, the integral is $O(A^{1/2} \Delta^{-100})$ unless $|bx-q| \asymp \frac{\Delta}{A}$, that is, 
\begin{equation} 
\label{eq:xqbapproximation}
 \Big|x-\frac{q}{b}\Big| \asymp \frac{\Delta}{M}. 
\end{equation} 
Here we are using the fact that $w_1$ has support on $[1,2]$, so that $t \asymp \Delta$.
Assuming \eqref{eq:xqbapproximation} holds, the integral can be asymptotically evaluated using stationary phase.  The conditions of \cite[Proposition 8.2]{BKY} are not quite met if $\Delta$ is small compared to $A$.  Instead, we may quote \cite[Lemma 5.5.6]{Huxley} (the reader may also try to derive a weighted version of \cite[Corollary 8.15]{IK}).  In this way, we have
\begin{equation*}
 \sum_{a=1}^{\infty} \frac{e(abx)}{a^{1/2-it}} f(a/A) = (t/e)^{it} \sum_{q \in \mz} \frac{P\big(\frac{t}{A(q-bx)}\big)}{(q-bx)^{1/2+it}} + O\Big(\frac{A^{1/2}}{\Delta^{3/2}}\Big),
\end{equation*}
where as in the proof of Lemma \ref{lemma:bigB}, $P(v)$ is a smooth function supported on $v \asymp 1$, satisfying $P^{(j)}(v) \ll 1$ (in fact, $P$ here is a constant multiple of $f$). 
 Hence
\begin{equation}
\label{eq:IABafterPoissonA}
I_{A,B} \ll \intR w_1\Big(\frac{t}{\Delta}\Big) \Big| \sum_{b,q} \frac{g(b/B)P\big(\frac{t}{A(q-bx)}\big) }{ b^{1/2+2iT-it} (q-bx)^{1/2+it}} \Big|^2 dt + O(AB \Delta^{-2}).
\end{equation}
The error term is satisfactory.

By analogy with \eqref{eq:IABfinalBoundBPoisson}, one might expect that $I_{A,B} \ll (\Delta + \frac{B \Delta}{A}) T^{\varepsilon}$, and note $\frac{B \Delta}{A} = \Delta \frac{B^2}{M} = M \frac{\Delta}{A^2}$ (and if $M = N =\sqrt{\Delta T}$, this is $N \frac{B^2}{T}$).
However, there are some differences so this guess is not entirely reliable, yet partially hints at the truth.  

Next we perform some simplifications, towards separation of variables.  Write $(b,q) = d$, and set $b = db_0$, $q = d q_0$ where now $(q_0, b_0) = 1$.  In the summation, we enforce the condition \eqref{eq:xqbapproximation} (of course $q/b = q_0/b_0$), which is redundant to the support of $P$.  We then obtain
\begin{equation*}
 I_{A,B} \ll \intR w_1\Big(\frac{t}{\Delta}\Big) \Big| \sum_d \frac{1}{d^{1+2iT}} \sum_{\substack{(b_0, q_0) = 1 \\ |x-\frac{q_0}{b_0}| \asymp M^{-1} \Delta}} \frac{g(db_0/B)P\Big(\frac{t}{Ad(q_0-b_0x)}\Big) }{ b_0^{1/2+2iT-it} (q_0-b_0x)^{1/2+it}} \Big|^2 dt + O(AB \Delta^{-2}).
\end{equation*}
By the same separation of variables method as in the proof of Lemma \ref{lemma:F0integralSimplified}, we can remove the weight function $P$ by a Mellin transform.  Let us say that the integration variable attached to $P$ is $v$.  Changing variables $t \rightarrow t-v$, and letting $2T' = 2T + v$, we obtain, for some smooth, nonnegative, compactly-supported function $w_3$,
\begin{equation}
\label{eq:IABboundwithRationalApproximation}
I_{A,B} \ll \max_{|T'-T| \ll \Delta^{\varepsilon}} \intR w_3\Big(\frac{t}{\Delta}\Big)
 \Big| \sum_d \frac{1}{d^{1+2iT'}} \sum_{\substack{(b_0, q_0) = 1 \\ |x-\frac{q_0}{b_0}| \asymp M^{-1} \Delta}} \frac{g(db_0/B) }{ b_0^{\frac12+2iT'-it} (q_0-b_0x)^{\frac12+it}} \Big|^2 dt + \frac{M}{\Delta^{2}}.
\end{equation}
Define the following coefficients $\alpha_b$ by
\begin{equation*}
\alpha_{b} = b^{-2iT'} \sum_d \frac{g(db/B)}{d^{1+2iT'}},
\end{equation*}
so with this notation
\begin{equation*}
I_{A,B} \ll 
\max_{|T'-T| \ll \Delta^{\varepsilon}} \intR w_3\Big(\frac{t}{\Delta}\Big)
 \Big| \sum_{\substack{(b_0, q_0) = 1 \\ |x-\frac{q_0}{b_0}| \asymp M^{-1} \Delta}} \frac{ \alpha_{b_0} }{ b_0^{1/2-it} (q_0-b_0x)^{1/2+it}} \Big|^2 dt + 
 \frac{M}{\Delta^{2}}
 .
\end{equation*}
  Note the trivial bound $\alpha_b \ll 1$.  We also have that
 \begin{equation}
  \label{eq:alphabbounds}
 \alpha_b \ll (\log T)^{-100}, \quad  \text{for } \quad  b \ll \frac{B}{(\log T)^{2/3+\varepsilon}},
 \end{equation}
by using a Mellin transform, and  the Vinogradov--Korobov bound (Lemma \ref{lemma:VinogradovKorobov}).
 
 We also choose a nonnegative $w_4$ so that $w_3 \leq w_4$ and $\widehat{w_4}$ has compact support (cf. \cite{Vaaler}).

Next we open the square and perform the integration.  For notational simplicity, we drop the max over $T'$ by assuming that $T'$ is chosen to maximize the expression.  Then we obtain
\begin{multline*}
I_{A,B} \ll \Delta 
\sum_{\substack{(b_1, q_1) = 1 \\ |x-\frac{q_1}{b_1}| \asymp M^{-1} \Delta}} \frac{|\alpha_{b_1}|}{(b_1 (q_1-b_1x))^{1/2}}
\\
\times \sum_{\substack{(b_2, q_2) = 1 \\ |x-\frac{q_2}{b_2}| \asymp M^{-1} \Delta}}  \frac{|\alpha_{b_2}|}{(b_2 (q_2-b_2x))^{1/2}}
|\widehat{w_4}|\Big(\frac{\Delta}{2\pi} \log\Big(\frac{b_2(q_1-b_1x)}{b_1(q_2-b_2x)}\Big) \Big),
\end{multline*}
plus $O(M \Delta^{-2})$.  
Using
\begin{equation*}
 \log\Big(\frac{b_2(q_1-b_1x)}{b_1(q_2-b_2x)}\Big) 
 =
 \log\Big(1+\frac{b_2 q_1 - b_1 q_2}{b_1(q_2 - b_2 x)}\Big),
\end{equation*}
and the compact support of $\widehat{w_4}$, we may further restrict attention to integers $b_1, b_2, q_1, q_2$ satisfying
\begin{equation}
\label{eq:q1b1q2b2}
 \Big|\frac{q_1}{b_1} - \frac{q_2}{b_2}\Big| \ll M^{-1},
\end{equation}
Therefore, we have
\begin{equation}
\label{eq:IABtworationalapproximations}
I_{A,B} \ll \frac{M}{\Delta^2} + M 
\sum_{\substack{(b_1, q_1) = 1 \\ |x-\frac{q_1}{b_1}| \asymp M^{-1} \Delta}} 
\thinspace \thinspace
\sumstar_{\substack{(b_2, q_2) = 1 \\ |x-\frac{q_2}{b_2}| \asymp M^{-1} \Delta}} 
\frac{|\alpha_{b_1} \alpha_{b_2}|}{b_1 b_2} ,
\end{equation}
where the star denotes that \eqref{eq:q1b1q2b2} holds.

Consider the diagonal contribution to \eqref{eq:IABtworationalapproximations}, that is, the terms with $q_1=q_2$ and $b_1 = b_2$.  These terms contribute to $I_{A,B}$ an amount that is
\begin{equation*}
 \ll M \sum_{b \ll B} \frac{|\alpha_b|^2}{b^2} \Big(1 + \frac{\Delta b}{M}\Big).
\end{equation*}
Using $\alpha_b \ll 1$, the part with $\frac{\Delta b}{M}$ contributes $O(\Delta  \log T)$. On the other hand,
using the trivial bound $\alpha_b \ll 1$ for $b \gg (\log T)^{100}$, and with $\alpha_b \ll (\log T)^{-100}$ for $b \ll (\log T)^{100}$, we obtain that 
\begin{equation*}
M \sum_{b \ll B} b^{-2} |\alpha_b|^2 \ll M (\log T)^{-100}
\ll M |\zeta(1+2iT)|^2 (\log T)^{-50}.
\end{equation*}
Thus the diagonal terms are bounded in accordance with Lemma \ref{lemma:bigA}.

Now consider the off-diagonal terms where $q_1/b_1 \neq q_2/b_2$.  
Let us restrict to say $b_1 \asymp X_1$, and $b_2 \asymp X_2$, where the $X_i$ range over dyadic segments, with $1 \ll X_i \ll B$.
We first claim the number of solutions in $(q_1,b_1)$ to \eqref{eq:xqbapproximation} with $b_1 \asymp X_1$ is $\ll 1 + \frac{\Delta X_1^2}{M}$.  To see this, note that the term ``1" accounts for a potential solution (this certainly cannot be improved for arbitrary $x$ since given $q,b$ there exist $x$ satisfying \eqref{eq:xqbapproximation}).  Once such a solution is found, we can estimate the number of others by noting
that the spacing between two reduced fractions of denominator of size $X_1$ is $\gg X_1^{-2}$, so the total number of such fractions that can be packed into an interval of length $L$ is bounded from above by $1 + L X_1^2$. This proves the claim. 
By a similar argument, for each choice of $(q_1, b_1)$, the number of $(q_2, b_2)$ with $0 < |\frac{q_1}{b_1} - \frac{q_2}{b_2}| \ll M^{-1}$ is $\ll \frac{X_1 X_2}{M}$.

Thus the bound on the off-diagonal terms is
\begin{equation*}
\ll \sum_{X_1, X_2 \text{ dyadic}} M \Big(1 + \frac{\Delta X_1^2}{M}\Big) \frac{1 }{X_1 X_2} \frac{X_1 X_2}{M} \ll \Big(\log T + \frac{\Delta B^2}{M}\Big) 
\log  T,
\end{equation*}
which completes the proof. 

\subsection{Proof of Lemma \ref{lemma:smallB}}
When $b$ is very small (and so $a$ is very large) then we treat the problem in a similar way to the early steps of \cite{DFI}.  We go back to the original definition \eqref{eq:IB0def}, open the square, and execute the integral, obtaining
 \begin{equation*}
  I_{B_0} = \Delta \sum_{a_1, a_2, b_1, b_2} 
  \frac{w_2\Big(\frac{a_1 b_1}{N}\Big)w_2\Big(\frac{a_2b_2}{N}\Big)g_0\Big(\frac{b_1}{B_0}\Big)g_0\Big(\frac{b_2}{B_0}\Big) e((a_1 b_1 - a_2 b_2)x)}{(a_1 a_2 b_1 b_2)^{1/2} (b_1/b_2)^{2iT}}
\widehat{w_1}\Big(-\frac{\Delta}{2\pi} \log \frac{a_1 b_1}{a_2 b_2} \Big).
 \end{equation*}
Our plan is to solve the equation $a_1 b_1 - a_2 b_2 = h \neq 0$ via Poisson summation in $a_2 \pmod{b_1}$.  The contribution to $I_{B_0}$ from $h=0$ is $O(\Delta \log T)$, by a trivial bound.  Let $I_{B_0}'$ be the contribution to $I_{B_0}$ from $h \neq 0$.
First we extract the greatest common divisor $d=(b_1, b_2)$ which necessarily divides $h$.  In this way we obtain
\begin{equation}
\label{eq:IB0expression}
I_{B_0}' =  \Delta \sum_{h \neq 0} e(hx) \sum_{d | h} \sum_{(b_1, b_2)=1} 
 \frac{g_0\Big(\frac{db_1}{B_0}\Big)g_0\Big(\frac{db_2}{B_0}\Big)}{(db_1)^{1/2 + 2iT} (db_2)^{1/2-2iT}} 
 R(b_1, b_2, h),
\end{equation}
where
\begin{equation*}
 R(b_1, b_2, h) = \sum_{a_1 b_1 - a_2 b_2 = \frac{h}{d}} \frac{w_2\Big(\frac{da_1 b_1}{M}\Big)w_2\Big(\frac{da_2b_2}{M}\Big)}{(a_1 a_2)^{1/2}}
 \widehat{w_1}\Big(-\frac{\Delta}{2\pi} \log \Big(1+ \frac{h/d}{a_2 b_2}\Big) \Big).
\end{equation*}
We next re-interpret the equation $a_1 b_1 - a_2 b_2 = \frac{h}{d}$ as a congruence $a_2 \equiv - \overline{b_2} \frac{h}{d} \pmod{b_1}$, and replace every occurrence of $a_1$ with $\frac{\frac{h}{d} + a_2 b_2}{b_1}$.
Therefore by Poisson summation in $a_2 \pmod{b_1}$, we have
\begin{multline*}
 R(b_1, b_2, h) = \frac{1}{b_1} \sum_{l \in \mz} e\Big(\frac{l \overline{b_2}(h/d)}{b_1}\Big) 
 \\
 \int_0^{\infty} 
 \frac{w_2\Big(\frac{h + d rb_2 }{M}\Big)w_2\Big(\frac{drb_2}{M}\Big)}{(\frac{h + dr b_2}{db_1} r)^{1/2}}
 \widehat{w_1}\Big(-\frac{\Delta}{2\pi} \log \Big(1+ \frac{h/d}{r b_2}\Big) \Big) e\Big(-\frac{l r}{b_1}\Big) dr.
\end{multline*}
Here the integral is of the form $\widehat{q}(y) = \int q(r) e(-ry) dr$ where $q$ is a function supported on $r \asymp \frac{M}{db_2} \gg \frac{M}{B_0}$, satisfying $q^{(j)}(r) \ll (\frac{M}{db_2})^{-j-1}$, and with $y = \frac{l}{b_1}$.  Thus $\widehat{q}(y) \ll_j (\frac{|y|M}{db_2})^{-j}$, by integration by parts. 
Now we have, with $y = l/b_1$, that
\begin{equation*}
 \frac{|y| M}{d b_2} = \frac{|l|M}{d b_1 b_2} \gg |l| \frac{M d}{B_0^2 } \gg |l| M T^{-2\alpha},
\end{equation*}
and since we assume $M \gg T^{1/2-\varepsilon}$, and $\alpha < 1/4$, if $l \neq 0$, then $\widehat{q}(l/b_1) \ll (lT)^{-100}$ which is very small.  For this reason, we only need to consider $l=0$.

Remark.  Here $b$ is small enough that Poisson summation conveniently leads only to the zero frequency, i.e., $l = 0$.  As in \cite[proof of Theorem 1]{DFI}, one can certainly consider larger values of $b$ which lead to nonzero frequencies.  However,  if $b$ is large compared to $a$, we should apply Poisson summation in $b$.  These arguments are used in the proof of Lemma \ref{lemma:ShiftedConvolutionWithBsummed}.


Now we return to the proof.  We have
\begin{multline}
\label{eq:IB0hnonzeroSum}
I_{B_0}' = \Delta \sum_{h \neq 0} e(hx) \sum_{\substack{d | h }} \sum_{(b_1, b_2)=1} 
 \frac{g_0\Big(\frac{db_1}{B_0}\Big)g_0\Big(\frac{db_2}{B_0}\Big)}{(db_1)^{1/2 + 2iT} (db_2)^{1/2-2iT}} 
 \\
 \times \frac{1}{b_1} \int_0^{\infty} 
 \frac{w_2\Big(\frac{h + dr b_2 }{M}\Big)w_2\Big(\frac{drb_2}{M}\Big)}{(\frac{h + dr b_2}{db_1} r)^{1/2}}
 \widehat{w_1}\Big(-\frac{\Delta}{2\pi} \log \Big(1+ \frac{h/d}{r b_2}\Big) \Big) dr,
\end{multline}
plus a small error of size $O(T^{-100})$.  
By changing variables $r \rightarrow \frac{r}{db_2}$, we obtain that \eqref{eq:IB0hnonzeroSum} equals
\begin{equation}
\label{eq:IB0expression2}
 \Delta \sum_{h \neq 0} e(hx) \sum_{\substack{d | h }} \frac{1}{d} \sum_{(b_1, b_2)=1} 
 \frac{g_0\Big(\frac{db_1}{B_0}\Big)g_0\Big(\frac{db_2}{B_0}\Big)}{b_1^{1 + 2iT} b_2^{1-2iT}} 
 \int_0^{\infty} 
 \frac{w_2\Big(\frac{  r}{M}\Big) w_2\Big(\frac{  r+h}{M}\Big)}{(r(r+h))^{1/2}}
 \widehat{w_1}\Big(-\frac{\Delta}{2\pi} \log\Big(1 + \frac{h}{r}\Big) \Big) dr.
\end{equation}
At this point, the integral is independent of $b_1$ and $b_2$.  Note that if we restrict attention to $d \geq D > 1$, then by a trivial bound, we obtain
\begin{equation*}
 \Delta \sum_{d \geq D} \frac{1}{d} \frac{M}{d \Delta} \log^2 B_0 \ll \frac{M}{D} \log^2 T.
\end{equation*}
We set $D = B_0^{1/2}$, so that this error term is satisfactory.

Next we claim that if $d \ll D = B_0^{1/2}$, then
\begin{equation}
\label{eq:b1b2sumevaluation}
 \sum_{(b_1, b_2)=1} 
 \frac{g_0\Big(\frac{db_1}{B_0}\Big)g_0\Big(\frac{db_2}{B_0}\Big)}{b_1^{1 + 2iT} b_2^{1-2iT}}  = \frac{|\zeta(1+2iT)|^2}{\zeta(2)} + O_C( (\log T)^{-C}),
\end{equation}
where $C > 0$ is arbitrarily large.  
For this, we simply apply a double Mellin transform, obtaining that the left hand side of \eqref{eq:b1b2sumevaluation} equals
\begin{equation*}
 \Big(\frac{1}{2\pi i} \Big)^2 \int_{(1)} \int_{(1)} \Big(\frac{B_0}{d}\Big)^{s_1 + s_2} \widetilde{g_0}(s_1) \widetilde{g_0}(s_2) \frac{\zeta(1 + 2iT + s_1) \zeta(1 - 2iT + s_2)}{\zeta(2+s_1 + s_2)} ds_1 d_2.
\end{equation*}
 Integrating by parts once, we have that for $\text{Re}(s) > 0$,
\begin{equation*}
 \widetilde{g_0}(s) = \frac{-1}{s} \int_0^{\infty} g_0'(x) x^s dx.
\end{equation*}
Since $g_0$ is identically $1$ for $0 \leq x <1$, $\widetilde{g_0}$ has a pole at $s=0$ (and nowhere else). 
Furthermore, $\text{Res}_{s=0} \widetilde{g_0}(s) = - \int_0^{\infty} g_0'(x) dx = 1$.
We move the $s_1$ integral to $\sigma_1 = - \frac{\alpha}{(\log T)^{2/3}}$, some $\alpha > 0$, followed by the same process for the $s_2$ integral.  
By the rapid decay of $\widetilde{g_0}$, the residues at $s_1 = -2iT$ and $s_2 = 2iT$ are very small.  
Using Lemma \ref{lemma:VinogradovKorobov} (Vinogradov--Korobov), we obtain that the left hand side of \eqref{eq:b1b2sumevaluation} equals
\begin{equation*}
 \frac{|\zeta(1+2iT)|^2}{\zeta(2)}  + O( (\log T)^{4/3} \exp\Big(-\beta \frac{\log B_0}{(\log T)^{2/3}}\Big),
\end{equation*}
for some constant $\beta > 0$.
Since we assume $B_0 \gg T^{\delta}$, this error term is $O( (\log T)^{-C})$ for $C > 0$ arbitrary.

Applying \eqref{eq:b1b2sumevaluation} to \eqref{eq:IB0expression2} and extending the $d$-sum back to all the divisors of $h$ (without making a new error term), we obtain for arbitary $C> 0$
\begin{multline*}
I_{B_0}' = \Delta \frac{|\zeta(1+2iT)|^2}{\zeta(2)} \int_0^{\infty} \sum_{h \neq 0} \sigma_{-1}(h) e(hx)
 \frac{w_2\Big(\frac{  r}{M}\Big) w_2\Big(\frac{  r+h}{M}\Big)}{(r(r+h))^{1/2}}
    \widehat{w_1}\Big(-\frac{\Delta}{2\pi } \log\Big(1 + \frac{h}{r}\Big) \Big) dr
    \\
  + O_C\Big( \frac{N}{(\log T)^{C}}\Big).
\end{multline*}
At this point we can use Taylor expansions (think of $\frac{h}{M} \ll \frac{1}{\Delta}$, even though this is not literally true), giving now
\begin{equation*}
 I_{B_0}' = \Delta \frac{|\zeta(1+2iT)|^2}{\zeta(2)} \int_0^{\infty} 
 \frac{w_2(r)^2}{r}
   \sum_{h \neq 0} \sigma_{-1}(h) e(hx) \widehat{w_1}\Big(-\frac{\Delta h}{2\pi M r }  \Big) dr
  + O\Big( \frac{M |\zeta(1+2iT)|^2}{ (\log T)^{100}} \Big).
\end{equation*}
By comparison to the expression for $MT_{OD}$ given by \eqref{eq:MTODmainterm}, we have shown that these terms equal $M \cdot S_{\Delta}$ plus a satisfactory error term.  We already showed that if $w_1$ has support on the positive reals, then $S_{\Delta}$ satisfies \eqref{eq:SDeltaProperty}.  If $w_1$ has support on $\mathbb{R}$, then we have $S_{\Delta} \ll 1$, by the remark following Lemma \ref{lemma:QxH}.  This completes the proof.

\subsection{Proof of Lemma \ref{lemma:ShiftedConvolutionWithBsummed}}
We return to the definition of $I_{A,B}$ as 
\begin{equation*}
 I_{A,B} = \frac{1}{M} \intR w_1\Big(\frac{t}{\Delta}\Big)
 \Big| \sum_{a,b} \frac{e(abx)}{a^{-it} b^{2iT-it}} f_1\Big(\frac{a}{A}\Big) 
 g_1\Big(\frac{b}{B}\Big)
 \Big|^2
 dt,
\end{equation*}
where $g_1(x) = x^{-1/2} g(x)$, and $f_1(x) = x^{-1/2} f(x)$.
We square out the sum and perform the $t$-integral, obtaining
\begin{equation*}
 I_{A,B} = \frac{\Delta}{M}
 \sum_{h \in \mathbb{Z}} e(hx)
 \sum_{\substack{a_1, a_2, b_1, b_2 \\ a_1 b_1 - a_2 b_2 = h}}
 \frac{\widehat{w_1}\Big(\frac{-\Delta}{2 \pi} \log\Big(1 + \frac{h}{a_2 b_2}\Big)\Big)  }{ (b_1/b_2)^{2iT}} 
 f_1\Big(\frac{a_1}{A}\Big) f_1\Big(\frac{a_2}{A}\Big) 
 g_1\Big(\frac{b_1}{B}\Big) g_1\Big(\frac{b_2}{B}\Big).
\end{equation*}
From the term $h=0$, we easily derive a bound of size $O(\Delta \log T)$, consistent with \eqref{eq:IABboundFromShiftedConvolutionWithBsummed}.
Let $I_{A,B}^{OD}$ denote the terms from $h \geq 1$.  By symmetry, it suffices to bound these terms.

Let $d = (a_1, a_2)$, which necessarily divides $h$, and change variables $a_i \rightarrow d a_i$, $h \rightarrow d h$, giving
\begin{equation*}
 I_{A,B}^{OD} = \frac{\Delta}{M}
\sum_{d }  \sum_{h \geq 1} e(dhx)  
 \sum_{\substack{(a_1, a_2) =1  \\  a_1 b_1 - a_2 b_2 = h}}
 \frac{\widehat{w_1}\Big(\frac{-\Delta}{2 \pi} \log\Big(1 + \frac{h}{a_2 b_2}\Big)\Big)  }{ (b_1/b_2)^{2iT}} 
 f_1\Big(\frac{da_1}{A}\Big) f_1\Big(\frac{da_2}{A}\Big) 
 g_1\Big(\frac{b_1}{B}\Big) g_1\Big(\frac{b_2}{B}\Big)
 .
\end{equation*}
Note that the terms with $h \gg \frac{M}{d \Delta} T^{\varepsilon}$ may be bounded by $O(T^{-100})$, from the rapid decay of $\widehat{w}$.  Next we interpet the equation $a_1 b_1 - a_2 b_2 = h$ as a congruence $b_2 \equiv -\overline{a_2} h \pmod{a_1}$, where we need to substitute $b_1 = \frac{a_2 b_2 + h}{a_1}$ wherever it appears.  To aid in this, we observe that
\begin{equation*}
 \Big(\frac{b_1}{b_2}\Big)^{2iT} = \Big(\frac{a_1}{a_2}\Big)^{-2iT} e^{2iT \log(1+\frac{h}{a_2 b_2})}.
\end{equation*}
Therefore, we have
\begin{multline*}
 I_{A,B}^{OD}
 = \frac{\Delta}{M} \Big[
\sum_{d \ll A } \sum_{
\substack{(a_1, a_2) = 1 \\ a_i \asymp A/d}} (a_1/a_2)^{2iT}
f_1\Big(\frac{da_1}{A}\Big) f_1\Big(\frac{da_2}{A}\Big)
\sum_{1 \leq h \ll \frac{M}{d\Delta} T^{\varepsilon}} e(dhx)  
 \\
 \sum_{b_2 \equiv - \overline{a_2} h \shortmod{a_1}}
 e^{-2iT \log(1+\frac{h}{a_2 b_2})}
 \widehat{w_1}\Big(\frac{-\Delta}{2 \pi} \log\Big(1 + \frac{h}{a_2 b_2}\Big)\Big)  
 g_1\Big(\frac{a_2 b_2 +h}{a_1 B}\Big) g_1\Big(\frac{b_2}{B}\Big)
 \Big]
 + O(T^{-100}).
\end{multline*}

We clean this up a bit by the approximations $g_1(\frac{a_2 b_2 + h}{a_1 B}) = g_1(\frac{a_2 b_2}{a_1 B}) + O(\frac{h}{a_1 B})$, and 
\begin{equation*}
 \widehat{w_1}\Big(\frac{-\Delta}{2 \pi} \log\Big(1 + \frac{h}{a_2 b_2}\Big)\Big) 
 =\widehat{w_1}\Big(\frac{-\Delta h }{2 \pi a_2 b_2} \Big) 
 + O\Big(\frac{\Delta h^2}{a_2^2 b_2^2} \Big(1+\frac{\Delta h}{a_2 b_2}\Big)^{-100}\Big).
\end{equation*}
The error terms in these approximations contribute at most $O(\frac{M}{\Delta})$ to $I_{A,B}$.  We also apply a dyadic partition of unity to the sum over $h$, with a generic piece denoted by $w_3(h/H)$, where $H$ runs over dyadic numbers $\ll \frac{M}{d \Delta} T^{\varepsilon}$.

Now define
\begin{equation*}
 U(v,y) = \widehat{w_1}\Big(\frac{-\Delta v }{2 \pi a_2 y} \Big) g_1\Big(\frac{a_2 y}{a_1 B}\Big) g_1\Big(\frac{y}{B}\Big) w_3\Big(\frac{v}{H}\Big),
\end{equation*}
which is supported on $v \asymp H$, $y \asymp B$, and therein satisfies the bounds
\begin{equation*}
 x^m y^n U^{(m,n)}(v,y) \ll_{m,n} 1,
\end{equation*}
uniformly in the auxiliary variables $a_1, a_2, d, \Delta, T, H,\dots$.  In the terminology of \cite{KiralPetrowYoung}, this means that $U$ forms a $1$-inert family of functions.

Therefore,
\begin{equation*}
 I_{A,B}^{OD} \ll \frac{\Delta}{M}
\sum_{d \ll A } \sum_{
\substack{(a_1, a_2) = 1 \\ a_i \asymp A/d}}
\sum_{H \text{ dyadic}}
|D_{a_1, a_2, d}| + O\Big(\frac{M}{\Delta}\Big),
\end{equation*}
where
\begin{equation*}
 D_{a_1, a_2, d} = \sum_{1 \leq h \ll \frac{M}{d\Delta} T^{\varepsilon}} e(dhx)  
 \sum_{b_2 \equiv - \overline{a_2} h \shortmod{a_1}}
 e^{-2iT \log(1+\frac{h}{a_2 b_2})}
U(h,b_2).
\end{equation*}

We now turn to the estimation of $D$, where the first step is applying Poisson summation in $b_2$.  We obtain
\begin{equation*}
 D_{a_1,a_2,d} = \sum_{1 \leq  h  \ll \frac{M}{d\Delta} T^{\varepsilon}} e(dhx)  
 \frac{1}{a_1}  \sum_{\nu \in \mz} e\Big(\frac{-\nu h \overline{a_2}}{a_1}\Big) 
 \intR e^{-2iT \log(1+\frac{h}{a_2 t}) - 2 \pi i \frac{\nu t}{a_1}}
 U(h,t) dt.
\end{equation*}
Write the phase above as
\begin{equation*}
 \phi(t) = - 2T \log\Big(1+\frac{h}{a_2 t}\Big) - 2 \pi  \frac{\nu t}{a_1},
\end{equation*}
and note
\begin{equation*}
 \phi'(t) = \frac{2T h}{t(a_2 t+h)}  - 2 \pi \frac{\nu}{a_1} = 
 \frac{2Th}{a_2 t^2} \Big(1 + O\Big(\frac{h}{a_2 B}\Big)\Big) - 2 \pi  \frac{\nu }{a_1}.
\end{equation*}
We have
\begin{equation*}
 \frac{h}{a_2 B} \ll \frac{T^{\varepsilon}}{\Delta},
 \qquad \text{and}
 \qquad B \frac{T h}{a_2 B^2} \gg \frac{T}{M} \gg T^{\delta}.
\end{equation*}
Taken together, these bounds show that $B |\phi'(t)| \gg T^{\delta}$ for all $t$ in the support of $G$ unless
\begin{equation}
\label{eq:nusize}
  \nu  \asymp  \frac{T H}{B^2},
\end{equation}
so in particular $\nu > 0$. 
By  \cite[Lemma 8.1]{BKY}, the condition $B |\phi'(t)| \gg T^{\delta}$ is sufficient to show that the integral is small, say $O(T^{-100})$, by repeated integration by parts, so from now on we suppose \eqref{eq:nusize} holds.  Our plan is to apply stationary phase analysis to the integral.  Note that
\begin{equation*}
- \phi''(t) = \frac{4 Th}{a_2 t^3} (1 + O(\Delta^{-1} T^{\varepsilon})) \asymp \frac{T H d}{M B^2}.
\end{equation*}
Moreover, $\phi'$ is monotone with a unique zero for $t \asymp B$ given by
\begin{align}
 t_0 = 
 \sqrt{\frac{h T a_1}{\pi \nu a_2}} [(1+\eta_0)^{1/2} - \eta_0^{1/2}],
\end{align}
where
\begin{equation*}
 \eta_0 = \frac{\pi h \nu }{4T a_1 a_2}.
\end{equation*}
Note that $\eta_0 \ll \Delta^{-2} T^{2 \varepsilon}$.  It is not hard to see that for $j \geq 1$, $\phi^{(j)}(t) \ll \frac{TH d}{M B^j}$.  We may then apply \cite[Main Theorem]{KiralPetrowYoung} with $Y = \frac{THd}{M}$, $Z=B$, $X_2 = H$.  Since $Y \gg T^{\delta}$ (using $hd \geq 1$ and $M \ll T^{1-\delta}$), the conditions to apply stationary phase are in place.  For forthcoming steps, we need to analyze $\phi(t_0)$.

Observe
\begin{equation*}
 \log\Big(1+\frac{h}{a_2 t_0}\Big) = \log(a_2 t_0 + h) - \log(a_2 t_0)
 = \log\Big(\frac{2T h a_1}{2 \pi \nu t_0}\Big) - \log(a_2 t_0),
\end{equation*}
which simplifies as $- 2 \log(\sqrt{1+\eta_0} - \sqrt{\eta_0})$.  Similarly, we have
\begin{equation*}
 2 \pi \frac{\nu}{a_1} t_0 = 4 T \sqrt{\eta_0}.
\end{equation*}
Therefore, we have
\begin{equation*}
 \begin{split}
  \phi(t_0) &= 4T  \log(\sqrt{1+\eta_0} - \sqrt{\eta_0})
 - 4 T \sqrt{\eta_0} (\sqrt{1+\eta_0} - \sqrt{\eta_0})
 \\
 &= -8T \eta_0^{1/2} (1 + c_{1} \eta_0^{1/2} + c_2 \eta_0 + c_{3} \eta_0^{3/2} + \dots),
 \end{split}
\end{equation*}
for certain constants $c_i$.  

By \cite[Main Theorem]{KiralPetrowYoung}, we conclude that
\begin{equation}
\label{eq:StationaryPhaseAsymptoticForb2integral}
 \intR e^{-2iT \log(1+\frac{h}{a_2 t}) - 2 \pi i \frac{\nu t}{a_1}}
 U(h,t)
 = \frac{B \sqrt{M}}{\sqrt{THd}} e^{-8iT \eta_0^{1/2}(1+\dots) } U_1(h) + O(T^{-200}),
\end{equation}
where $U_1(h) = U_1(h, \nu, a_1, a_2, \dots)$ 
satisfies the derivative bounds
\begin{equation}
\label{eq:U1derivativebounds}
 v^m U_1^{(m)}(v) \ll 1,
\end{equation}
with uniformity in all auxiliary variables.  One could alternatively apply, say,  \cite[Lemma 5.5.6]{Huxley} to obtain an asymptotic in \eqref{eq:StationaryPhaseAsymptoticForb2integral}, with a weaker (but sufficient) error term, and a weight function of the form $c U(h, t_0)$, for some constant $c$.  One would then need to manually check \eqref{eq:U1derivativebounds}, which is a somewhat tedious exercise.  In the same vein, \cite[Proposition 8.2]{BKY} develops a full asymptotic expansion for \eqref{eq:StationaryPhaseAsymptoticForb2integral}, but does not directly give the bounds on $U_1(h)$ with respect to $h$ (however, \cite[Main Theorem]{KiralPetrowYoung} builds on this work in a crucial way).

We insert \eqref{eq:StationaryPhaseAsymptoticForb2integral} back into $D$, obtaining
\begin{equation*}
 D_{a_1,a_2,d} =
 \sum_{\nu \asymp \frac{T H}{B^2}} \frac{B \sqrt{M}}{a_1 \sqrt{THd}}
 \sum_{h \asymp H} e(dhx)  
   e^{-8iT \eta_0^{1/2}(1+\dots) }  U_1(h) + O(T^{-100}).
\end{equation*}
To track our progress, note that the trivial bound applied to $D$ now shows $D \ll\frac{M \sqrt{T}}{\Delta^{3/2}}  T^{\varepsilon}$, while the original trivial bound is $D \ll \frac{MB}{\Delta A} T^{\varepsilon}$ (assuming $B \gg A$).  In the case of interest to us, when $A \approx \sqrt{\Delta}$ and $B \approx \sqrt{T}$, these bounds are equalized.  However, the new formula for $D$ is more accessible to understand the dependence on $h$.

Consider the exponential sum
\begin{equation*}
 \sum_{h \asymp H} e^{-8iT \eta_0^{1/2}(1+\dots) + 2\pi i dhx } U_1(h),
\end{equation*}
where the phase is of the form
\begin{equation*}
 \varphi(h) = - 8 T \sqrt{\gamma_0 h} (1 + c_{1} \sqrt{\gamma_0 h} + c_2 (\gamma_0 h) + \dots) + 2\pi dhx,
\end{equation*}
where $h\gamma_0 = \eta_0$.  For each $j \geq 2$, we have
\begin{equation*}
\varphi^{(j)}(h) =  \frac{T \sqrt{\gamma_0 h}}{h^j}
(d_0^{(j)} + d_{1}^{(j)} (\gamma_0 h)^{1/2} + \dots ),
\end{equation*}
for certain constants, with $d_0^{(j)} \neq 0$.  We conclude that
\begin{equation*}
 |\varphi^{(j)}(h)| \asymp \frac{T \sqrt{\gamma_0 H}}{H^j} \asymp \frac{dHT/M}{H^j}.
\end{equation*}
By partial summation, and Weyl's method (see \cite[Corollary 8.5]{IK} and check that the frequency $dHT/M$ is larger than the length $H$ of summation, since, $\frac{dHT}{M} \gg H T^{\delta}$), we conclude that
\begin{equation*}
 D_{a_1, a_2, d} \ll T^{\varepsilon} 
 \frac{d}{A} 
 \frac{TH}{B^2}
 \frac{B \sqrt{M}}{\sqrt{THd}} H^{1/2} (dHT/M)^{1/6} + T^{-100}
 \ll 
 T^{\varepsilon}
 \Big(\frac{T}{\Delta}\Big)^{2/3} \Big(\frac{M}{\Delta}\Big)^{1/2}.
\end{equation*}
Hence,
\begin{equation*}
\frac{\Delta}{M}
\sum_{d \ll A } \sum_{
\substack{(a_1, a_2) = 1 \\ a_i \asymp A/d}}
\sum_{H \text{ dyadic}}
|D_{a_1, a_2, d}|
\ll T^{\varepsilon} \frac{A^2}{\Delta}  \frac{(\Delta T)^{1/4}}{M^{1/2}} (\Delta T)^{1/2} \Big(\frac{\Delta}{T}\Big)^{1/12}.
\end{equation*}
This completes the proof of Lemma \ref{lemma:ShiftedConvolutionWithBsummed}.

\section{The smallest values of $\Delta$}
\label{section:verysmallDelta}
In this section we show Proposition \ref{prop:verysmallDelta}, that is, 
\begin{equation}
\label{eq:Fintegraltverysmall}
\int_{|t-T| \ll T^{1/100}} |F(it)|^2 dt \ll (\log T)^{\varepsilon}.
\end{equation}

The analysis from Section \ref{section:initialdevelopments} does not cover all of this range because Stirling's asymptotic is no longer applicable (one of the gamma factors making up $\gamma_{V_T}(1/2+it)$ is evaluated at imaginary part $\approx |T-t|$).  Instead, we go back to the original definition \eqref{eq:Fformula1}.  Obviously,
\begin{equation*}
\intR |\widetilde{\psi}(1/2 + it \pm iT)|^2 dt \ll 1,
\end{equation*}
so that by Cauchy--Schwarz, we have
\begin{equation}
\label{eq:KDeltaSum}
 \int_{|t-T| \ll T^{1/100}} |F(it)|^2 dt \ll 1 + \frac{1}{|\zeta(1+2iT)|^2 } \sum_{\substack{\Delta \ll T^{1/100} \\ \text{dyadic}}} K(\Delta, T, x),
\end{equation}
where
\begin{equation*}
K(\Delta, T, x) =  \int_{|t-T| \asymp \Delta} \Big|
\sum_{n \geq 1} \frac{\tau_{iT}(n)
e(nx)}{n^{1/2 + it}} 
k(n,t,T)
 \Big|^2 dt,
\end{equation*}
and where with $\sigma > -\frac12$, we have
\begin{equation*}
 k(n,t,T) = \frac{1}{2 \pi i} \int_{(\sigma)} \widetilde{\psi}(-u) n^{-u}
\cosh(\pi T/2) \gamma_{V_T}(1/2 + it+u) du.
\end{equation*}
When $\Delta \asymp 1$, we need to re-interpret $|t-T| \asymp \Delta$ as $|t-T| \ll 1$, but we do not make new notation for this range. 

Next we split the $n$-sum into three intervals: 
$n \ll N_0^{-}$, 
$N_0^{-} \ll n \ll N_0^{+}$, 
and 
$n \gg N_0^{+}$, where
\begin{equation*}
 N_0^{-} = \frac{(\Delta T)^{1/2}}{(\log T)^2}, \qquad 
 N_0^{+} = (\Delta T)^{1/2}(\log T)^{\varepsilon}.
\end{equation*}
Let us call these three intervals $\mathcal{N}_i$, $i=1,2,3$, respectively.
To simplify the analysis, we suppose that these three ranges are detected by a smooth partition of unity, say $\gamma_1 + \gamma_2 + \gamma_3 = 1$, with $\gamma_i$ respective to $\mathcal{N}_i$.

Our goal is to relate $K(\Delta, T, x)$ to expressions to which we may apply the results of Section \ref{section:smallmediumDelta}.  To this end, we have
\begin{mylemma}
\label{lemma:QWERTYbound}
 Suppose that $\Delta \gg (\log T)^{\varepsilon}$.  Then
 \begin{equation}
 \label{eq:QWERTYbound1}
 K(\Delta, T, x) \ll  \int_{|t-T| \asymp \Delta} (\Delta T)^{-\half} 
\Big| \sum_{N_0^{-} \ll n \ll N_0^{+}} \frac{\tau_{iT}(n)
e(nx) \gamma_2(n)}{n^{1/2 +it}} \Big|^2  dt + \frac{|\zeta(1+2iT)|^2}{(\log T)^{2}}.
\end{equation}
 If $\Delta \ll (\log T)^{\varepsilon}$, then
\begin{equation}
\label{eq:QWERTYbound2}
 K(\Delta, T, x) \ll_{\varepsilon}  \int_{|t-T| \ll (\log T)^{\varepsilon}} T^{-\half} 
\Big| \sum_{N_0^{-} \ll n \ll N_0^{+}} \frac{\tau_{iT}(n)
e(nx) \gamma_2(n)}{n^{1/2 +  it}} \Big|^2  dt + \frac{|\zeta(1+2iT)|^2}{(\log T)^{2}}.
\end{equation}
\end{mylemma}
The error terms here contribute $\ll (\log T)^{-1}$ to \eqref{eq:Fintegraltverysmall}, after summing over the $O(\log T)$ values of $\Delta$ in \eqref{eq:KDeltaSum}, which is satisfactory.

\begin{proof}

By Cauchy's inequality, it suffices to bound the three ranges $\mathcal{N}_i$ separately.  For $n \ll N_0^{-}$ (that is, $i=1$), we take $\sigma_1 = -\frac14$, and for $n \gg N_0^{+}$ ($i=3$), we take $\sigma_3 = 1/\varepsilon^2$ (large).  In each $n$-range, we then move the $u$-integral outside and apply Cauchy--Schwarz in the following form (similarly to \eqref{eq:psiseparation1}  above):
\begin{multline}
\label{eq:KDeltaSeparationofVariables}
 \int_t \Big| \sum_{n \in \mathcal{N}_i} \frac{\tau_{iT}(n)
e(nx) \gamma_i(n)}{n^{1/2 + it}} \frac{1}{2 \pi i} \int_{(\sigma_i)} \widetilde{\psi}(-u) n^{-u}
\cosh(\pi T/2) \gamma_{V_T}(1/2 + it+u) du \Big|^2 dt
\\
\ll
\int_t \Big[ \Big(\int_{(\sigma_i)} |\widetilde{\psi}(-v)| \cosh(\pi T) |\gamma_{V_T}(1/2+it+v)|^2 dv \Big) 
\\
\times \Big( \int_{(\sigma_i)} |\widetilde{\psi}(-u)| \Big| \sum_{n \in \mathcal{N}_i} \frac{\tau_{iT}(n)
e(nx) \gamma_i(n)}{n^{1/2 + it+u}} \Big|^2 du  \Big) \Big] dt.
\end{multline}
We next argue that
\begin{equation}
\label{eq:psiMellingammaVTIntegralBound}
 \int_{(\sigma)} |\widetilde{\psi}(-v)| \cosh(\pi T) |\gamma_{V_T}(1/2+it+v)|^2 dv
 \ll (1 + |t-T|)^{\sigma - \half} (1 + |t+T|)^{\sigma-\half}.
\end{equation}
By Stirling's approximation, the left hand side of \eqref{eq:psiMellingammaVTIntegralBound}  is bounded by
\begin{equation*}
 \intR |\widetilde{\psi}(-\sigma - iy)| (1 + |t+T+y|)^{\sigma-\half} (1 + |t-T+y|)^{\sigma-\half} dy.
\end{equation*}
The main range of the integral is $|y| \leq 1 + \half |t-T|$, which by a trivial bound leads to the claimed bound.  If $|y| \geq 1 + \half |t-T|$, then the rapid decay of $\widetilde{\psi}$ can be used to show that this part of the integral is negligible by comparison.  
Hence
\begin{equation}
\label{eq:tinyDeltaNiceExpressionAA}
 K(\Delta, T, x) \ll \sum_{i=1}^{3} \int \limits_{|t-T| \asymp \Delta} (\Delta T)^{\sigma_i-\half} 
 \intR |\widetilde{\psi}(-\sigma_i - i y)|
 \Big| \sum_{ n \in \mathcal{N}_i} \frac{\tau_{iT}(n)
e(nx) \gamma_i(n)}{n^{1/2 +\sigma_i + it + iy}} \Big|^2 dy dt.
\end{equation}

To focus on the important parts here, we shall apply the mean value theorem for Dirichlet polynomials (combined with Lemma \ref{lemma:diagonaltermsum}) in some ranges where this is acceptable.  In \eqref{eq:tinyDeltaNiceExpressionAA}, with $i =1$ (so $\sigma_1 = -1/4$), we have the bound
\begin{equation*}
 \sum_{n \ll N_0^{-}} \frac{|\tau_{iT}(n)|^2}{n^{1/2}} \frac{(\Delta + n)}{(\Delta T)^{3/4} } \ll \frac{|\zeta(1 + 2iT)|^2 \log T}{(\Delta T)^{3/4}} (\Delta (N_0^{-})^{1/2} + (N_0^{-})^{3/2}) \ll \frac{|\zeta(1+2iT)|^2}{(\log T)^{2}},
\end{equation*}
consistent with Lemma \ref{lemma:QWERTYbound}.
A similar (even stronger) bound holds for $\mathcal{N}_3$, and we omit the details as it is even easier than the above case (it saves an arbitrary power of $\log T$).  Therefore,
\begin{multline}
\label{eq:tinyDeltaNiceExpressionBB}
 K(\Delta, T, x) \ll  \int \limits_{|t-T| \asymp \Delta} (\Delta T)^{-\half} 
 \intR |\widetilde{\psi}( - i y)|
 \Big| \sum_{N_0^{-} \ll  n \ll N_0^{+}} \frac{\tau_{iT}(n)
e(nx) \gamma_2(n)}{n^{1/2  + it + iy}} \Big|^2 dy dt 
\\
+ \frac{|\zeta(1+2iT)|^2}{(\log T)^{2}}.
\end{multline}

If $\Delta \gg (\log T)^{\varepsilon}$, then we can use an argument as in the proof of Lemma \ref{lemma:F0integralSimplified} to truncate the $y$-integral at $\frac{1}{100} \Delta$, and then over-extend the $t$-integral slightly, which is valid by the positivity of the integrand.  The error term obtained in this way saves an arbitrarily large power of $\log T$ which is acceptable for these values of $\Delta$.  This proves \eqref{eq:QWERTYbound1}.

Next suppose $\Delta \ll (\log T)^{\varepsilon}$.  In this situation, we truncate the $y$-integral at $(\log T)^{\varepsilon}$, which by the rapid decay of $\widetilde{\psi}$ introduces an error term that saves an arbitrarily large power of $\log T$ over the trivial bound, which is satisfactory.  We then use a simple over-estimate in the following shape: 
\begin{equation*}
(\Delta T)^{-1/2} \int_{|t-T| \asymp \Delta}  |f(t)|^2 dt \leq T^{-1/2} \int_{|t-T| \ll (\log T)^{\varepsilon}} |f(t)|^2 dt.
\end{equation*}
At this point, the $t$-integral has the same length as the $y$-integral, so we can again change variables $t \rightarrow t-y$ and double the length of $t$-integration to separate the variables $t$ and $y$.  This proves \eqref{eq:QWERTYbound2}.
\end{proof}

Now we are ready to prove the bound \eqref{eq:Fintegraltverysmall}. 
First, we claim
\begin{equation}
\label{eq:claimedBoundKDelta}
 \sum_{\substack{(\log T)^{100} \ll \Delta \ll T^{1/100} \\ \text{dyadic}}} 
 \int_{|t-T| \asymp \Delta}  \frac{(\Delta T)^{-\half}}{|\zeta(1+2iT)|^2}
\Big| \sum_{N_0^{-} \ll n \ll N_0^{+}} \frac{\tau_{iT}(n)
e(nx) \gamma_2(n)}{n^{1/2 +it}} \Big|^2  dt
 \ll (\log T)^{\varepsilon}.
\end{equation}
In order to apply the results from Section \ref{section:smallmediumDelta}, we apply a dyadic partition of unity within the $n$-sum, and apply Cauchy's inequality.  
Since $\log(N_0^+/N_0^-) \ll \log \log T$, there are at most $O(\log \log T)$ such dyadic pieces to consider.
We then open the divisor function $\tau_{iT}(n) = \sum_{ab=n} (a/b)^{iT}$, and apply a partition of unity to the $a$- and $b$-sums, leading to a similar decomposition as in Lemma \ref{lemma:IDeltaDyadicDecomposition}.  One small difference is that now $n \asymp M$, where $N_0^- \ll M \ll N_0^+$.  We apply Lemma \ref{lemma:smallB} to satisfactorily bound $I_{B_0}$, choosing $B_0 = T^{13/56}$.  The factor $(\log T)^{\varepsilon}$ arises because $\frac{N_0^+}{(\Delta T)^{1/2}} \ll (\log T)^{\varepsilon}$.  Now we consider $I_{A,B}$ where $B \gg T^{\frac{13}{56}}$.  By \cite[Theorem 2.9]{GK}, with $q=3, Q=8$, and partial summation, we have
\begin{equation*}
 \sum_{b \asymp B} \frac{e(abx)}{b^{1/2+2iT-it}} \ll B^{1/2}  \frac{T^{1/30}}{B^{1/6}}.
\end{equation*}
Therefore, trivially summing over $a$ and integrating over $t$, we obtain
\begin{equation*}
 I_{A,B} \ll AB \Delta \frac{T^{1/15}}{B^{1/3}} \ll AB T^{\frac{1}{100}+\frac{1}{15}-\frac{13}{168}} = M T^{-1/1400}.
\end{equation*}
Since this gives a power saving, it is easily seen to be satisfactory.  This shows \eqref{eq:claimedBoundKDelta}.

Next, consider the terms with $(\log T)^{\varepsilon} \ll \Delta \ll (\log T)^{100}$, as well as $\Delta \ll (\log T)^{\varepsilon}$.  In both cases, we reduce to $O((\log T)^{\varepsilon})$ instances of expressions of the form
\begin{equation*}
 \frac{(\log T)^{\varepsilon}}{M |\zeta(1+2iT)|^2}
 \int_{|t-T| \ll (\log T)^{100}} 
 \Big|
\sum_{n} \frac{\tau_{iT}(n)
e(nx) w_2(n/M)}{n^{1/2 +it}} \Big|^2  dt.
 \end{equation*}
Again, we decompose this as in Lemma \ref{lemma:IDeltaDyadicDecomposition}, with $B_0 = T^{13/56}$.
The $I_{A,B}$ parts are bounded as in the previous paragraph, giving a power saving.
The $I_{B_0}$ part is bounded with Lemma \ref{lemma:smallB}.  The remark following Lemma \ref{lemma:smallB} indicates that in the present setting, where $w_1$ has support on, say, $[-10,10]$, we may only claim $S_{\Delta} \ll 1$.  This is acceptable, because there are only $O((\log T)^{\varepsilon})$ such instances of these expressions.  This completes the proof.

\section{Corollary \ref{coro:signchanges}}
\label{section:shiftedQUE}
\subsection{Shifted QUE}
The main tool in proving Corollary \ref{coro:signchanges} is the following more general version of Theorem \ref{thm:QUE}.  Let $a$ be a fixed (small) real number, and define the shifted restricted QUE integral:
\begin{equation*}
 I_{\psi,a}(x,T) = \int_0^{\infty} \psi(y)
E_T^*(x+iy) \psi(y(1+T^{-1} a)) E_T^*(x+iy(1+T^{-1} a)) 
\frac{dy}{y}.
\end{equation*}

 \begin{mytheo}
\label{thm:shiftedQUE}
Let conditions be as in Theorem \ref{thm:QUE}.  Then 
\begin{equation}
\label{eq:shiftedQUE}
 I_{\psi,a}(x,T)  =   \frac{3}{\pi} \log(1/4 + T^2) \int_0^{\infty} \psi^2(y) \Big[J_0(a) + B_q(1,T) J_0\Big(\Big|\frac{\theta T}{y} + i a \Big| \Big)\Big] \frac{dy}{y} 
  + O((\log T)^{35/36+\varepsilon}).
\end{equation}
The implied constant depends continuously on $a$, and is independent of $x$.
\end{mytheo}
To deduce Corollary \ref{coro:signchanges}, we let $a = a_0 = 3.831\dots$ be the location of the global minimum of $J_0(x)$ for $x \geq 0$.  Then $J_0(a_0) = -0.4027\dots$, and $|J_0(x)| < |J_0(a_0)|$ for all $x > a_0$.  Hence $|J_0(|y^{-1} \theta T + ia|)| \leq |J_0(a_0)|$ (with equality iff $\theta = 0$, i.e. $x$ is rational).  We also claim that $|B_q(1,T)| < 1$, for all $q \geq 2$, which we prove in Section \ref{section:BqBound} below.  Of course, $B_q(1,T) = 1$ for $q=1$.
In any event, we may deduce that $I_{\psi, a_0}(x,T) < 0$ for all $x \in \mathbb{R}$ and $T$ sufficiently large.  Therefore, there must exist $y$ in the support of $\psi$ so that $E_T^*(x+iy)$ and $E_T^*(x+iy(1+T^{-1} a_0))$ have opposite signs, and hence also a zero in-between.

Now we embark on the proof of Theorem \ref{thm:shiftedQUE}.  Since many of the details are nearly identical to those in the proof of Theorem \ref{thm:QUE}, we will be brief in those instances.  We also refer to \cite{JungYoung} for more in-depth explanations in certain locations.  In place of \eqref{eq:IpsiFitSquared}, we instead have
\begin{equation*}
I_{\psi,a}(x,T) = \frac{1}{2 \pi} \intR |F(it)|^2 \Big(1 + \frac{a}{T}\Big)^{-it} dt,
\end{equation*}
for which see \cite[(5.5)]{JungYoung}.
Therefore, any already-proven upper bound on an integral of $|F(it)|^2$ over some interval immediately implies the same bound here.  In particular, the proof of Theorem \ref{thm:QUE} gives that
\begin{equation*}
I_{\psi,a}(x,T) = \frac{1}{2 \pi} \int_{|t| \leq T - \eta T} |F(it)|^2 \Big(1 + \frac{a}{T}\Big)^{-it} dt + O(\eta^{1/2} \log T),
\end{equation*}
where $\eta$ is some fixed small negative power of $\log T$.  We may also freely replace $F$ by $F_0$ in the above approximation.  

Now we follow some of the steps from \cite{JungYoung} to handle this new weight $(1+T^{-1} a)^{-it}$.
By \cite[Lemma 5.4]{JungYoung}, we have
\begin{equation*}
\Big(1 + \frac{a}{T}\Big)^{-it} = \sum_{\ell = 0}^{L} \binom{-it}{\ell} \Big(\frac{a}{T} \Big)^{\ell} + O(10^{-L}),
\end{equation*}
uniformly for $|t| \leq T$ and $a$ in some fixed compact set.   We pick $L \asymp \log \log T$, with an implied constant large enough so that $10^{-L} \log T \ll \eta^{1/2} \log T$.  Hence
 \begin{equation*}
I_{\psi,a}(x,T) = \sum_{\ell = 0}^{L} a^{\ell} \frac{1}{2 \pi} \int_{|t| \leq T - \eta T} 
\binom{-it}{\ell} T^{-\ell} 
|F_0(it)|^2 dt + O(\eta^{1/2} \log T).
\end{equation*}
For $\ell \geq 1$, the contribution from $|t| \leq T^{1/2}$ (say), is $O(T^{-1/2+\varepsilon})$, On the other hand, when $|t| \geq T^{1/2}$, then we have
\begin{equation*}
\binom{-it}{\ell} = \frac{(-it)^{\ell}}{\ell !} ( 1 + O(|t|^{-1} \ell^2)),
\end{equation*}
 which is a small error term.  After applying this approximation, we may then extend the integral back to $|t| \leq T^{1/2}$ with an error that is absorbed by the existing error term.  Thus,
 \begin{equation*}
 I_{\psi,a}(x,T) = \sum_{\ell = 0}^{L} \frac{(-ia)^{\ell}}{\ell !} \frac{1}{2 \pi} \int_{|t| \leq T - \eta T} 
\Big(\frac{t}{T}\Big)^{\ell} 
|F_0(it)|^2 dt + O(\eta^{1/2} \log T).
 \end{equation*}

Now we carry through the arguments from Section \ref{section:largeDelta}.  Essentially the only difference is that $J(m,n)$ is altered by the presence of the factor $(t/T)^{\ell}$ in \eqref{eq:Jdef}.  It is easy to see that the $k$-th derivative bound on the new $J$ is altered by multiplication by at most $O(\ell^k) \ll (\log \log T)^k$.  These factors are easily absorbed by a factor $(\log T)^{\varepsilon}$.  We therefore obtain the same error term stated in \eqref{eq:F0integralErrorTerm}.  The main term changes in an interesting way, however, so we now focus on that aspect.  

The analog of the formula \eqref{eq:JthetaMellinPrimeAtZero} for $\widetilde{J}_{\theta}'(0)$ is now of the form
\begin{equation*}
\frac{\pi^2}{2} \log T \int_0^{\infty} \psi^2(y) 
 \Big[ \int_{|t| \leq T} \Big(\frac{t}{T}\Big)^{\ell} (T^2-t^2)^{-1/2}  e^{i \theta \frac{\sqrt{T^2-t^2}}{y}} \frac{dt}{2 \pi} \Big] \frac{dy}{y}.
\end{equation*}
 For $\ell$ odd this integral vanishes, while for $\ell$ even we have from \cite[8.411.6]{GR} and gamma function identities, that
 \begin{equation*}
  \int_{|t| \leq T} \Big(\frac{t}{T}\Big)^{2\ell} (T^2-t^2)^{-1/2}  \cos\Big(\theta \frac{\sqrt{T^2-t^2}}{y} \Big) dt = \pi \frac{(2\ell)!}{\ell!} 2^{-2\ell} \Big(\frac{\theta T}{2 y}\Big)^{-\ell} J_{\ell}\Big(\frac{\theta T}{ y}\Big).
 \end{equation*}
This has the effect that the term $J_0(y^{-1} \theta T)$ appearing in Theorem \ref{thm:QUE} is now replaced by
\begin{equation*}
S:=\sum_{\ell=0}^{\infty} \frac{ (-1)^{\ell} a^{2\ell} }{(2 \ell)!} 
\frac{(2\ell)!}{\ell!} 2^{-2\ell} (u/2)^{-\ell} J_{\ell}(u),
\end{equation*}
where $u = y^{-1} \theta T$.
 Technically, the sum we obtain is truncated at $2 \ell \leq L$, but may be extended to $\infty$ with an error that is absorbed by the existing error term.  We claim
 \begin{equation*}
 S = J_0(|u + i a|).
 \end{equation*}
To see this, we insert the power series expansion of $(u/2)^{-\ell} J_{\ell}(u)$, which gives
\begin{equation*}
S = \sum_{\ell=0}^{\infty} \frac{ (-1)^{\ell} a^{2\ell} }{\ell!} 
 2^{-2\ell} \sum_{k=0}^{\infty}  \frac{ (-1)^{k} u^{2k} }{k! \Gamma(k+\ell+1)} 
 2^{-2k}.
\end{equation*} 
Letting $k+\ell = m$, this neatly simplifies with the binomial theorem to give
\begin{equation*}
S = \sum_{m=0}^{\infty} \frac{(-1)^m}{2^{2m} m!^2} (a^2 + u^2)^m,
\end{equation*} 
 which is precisely the power series expansion of $J_0(\sqrt{a^2 + u^2})$, as claimed.  This completes the proof of Theorem \ref{thm:shiftedQUE}.
 
\subsection{Bounding $B_q(1,T)$}
\label{section:BqBound} 
 It is clear that $B_q(1,T) \ll q^{-1+\varepsilon}$, but it is pleasant to claim a clean bound valid for all $q$, which allows for the uniform statement in Corollary \ref{coro:signchanges} (with no exceptions near rationals with small denominators).

If $q = p^{q_p}$ with $p$ prime and $q_p \geq 1$, then
\begin{equation*}
B_q(1,T) = \frac{Y+X}{Z+X},
\end{equation*}
where $Y = - (p-1)^{-1} p^{-(q_p-1)} \tau_{iT}(p^{q_p-1})^2$, $X = \sum_{j=q_p}^{\infty} p^{-j} |\tau_{iT}(p^j)|^2$, and $Z = \sum_{j=0}^{q_p-1} p^{-j} |\tau_{iT}(p^j)|^2$.  It is easy to see $B_q(1,T) < 1$, since $Y \leq 0 < Z$.  It is also easy to see $-Y < Z$, whence $B_q(1,T) > \frac{-Z+X}{Z+X} \geq -1$.  Thus, each local factor making up $B_q(1,T)$ has absolute value strictly less than $1$, and so $|B_q(1,T)| < 1$ for all $q \geq 2$.

\section{Maass cusp forms}
\label{section:conjecture}
Here we provide some discussion of Conjecture \ref{conj:QUEcuspCase}.  Begin by writing the Fourier expansion in the form 
\begin{equation*}
u_j(x+iy) = \rho_j(1) \sum_{n \neq 0} \frac{\lambda_{u_j}(n) e(nx)}{|n|^{1/2}} V_{T_j}(2 \pi |n| y).  
\end{equation*}
Here $\rho_j(1)$ is determined by the normalization of $u_j$.
The initial developments from Section \ref{section:initialdevelopments} carry over with minimal changes.   We expect that the main term comes from the interval $|t|  \leq T - o(T)$, which is the range of $t$ studied in Section \ref{section:largeDelta}.  Moreover, we expect that there is cancellation in shifted convolution sums with Hecke eigenvalues, and so the main term should come from the diagonal, as in \eqref{eq:MainTermDef}.  It is easy to formally carry through the argument from Section \ref{section:diagonaltermsGeneralx}.  The main difference is that $\mathcal{L}_d(s)$ now takes the form
\begin{equation*}
 \mathcal{L}_d(s) = \sum_{n=1}^{\infty} \frac{\lambda_{u_j}(dn)^2 \chi(n)}{n^s}.
\end{equation*}
This Dirichlet series has a pole iff $\chi$ is principal, and moreover the pole is simple (for comparison, in the Eisenstein series case, the pole is double).

The contribution from $\chi = \chi_0$ is then analogous to \eqref{eq:MTx0}, but with
\begin{equation*}
 Z_q(s) = \sum_{n=1}^{\infty} \frac{\lambda_{u_j}(n)^2}{n^s} \frac{\mu(q/(n,q))}{\varphi((q/(n,q))} = 
 Z_1(s) B_q(s,u_j),
\end{equation*}
say, where
 \begin{equation}
\label{eq:BqDefCuspFormCase}
 B_q(s,u_j) =  \prod_{p^{q_p} || q} \frac{\sum_{k=0}^{\infty} 
\frac{\lambda_{u_j}(p^k)^2}{p^{ks}} \frac{\mu(p^{q_p}/(p^k,p^{q_p}))}{\varphi((p^{q_p}/(p^k,p^{q_p}))}
}{\sum_{k=0}^{\infty} 
\frac{\lambda_{u_j}(p^k)^2}{p^{ks}} }.
\end{equation}
By standard Rankin--Selberg calculations, we have that $\frac{\rho_j(1)^2}{\cosh( \pi T_j)} \text{Res}_{s=1} Z_1(s)$ equals an absolute constant.  We indirectly work out the constant of proportionality at the end.

In Section \ref{section:residue}, we bounded $\widetilde{J_{\theta}}(0)$, and asymptotically evaluated $\sum_{\pm} \widetilde{J}_{\pm \theta}'(0)$.  In the cusp form case there is no reason to analyze $\widetilde{J}_{\pm \theta}'(0)$, since the Dirichlet series only has a single pole, and instead we need to asymptotically evaluate $\sum_{\pm} \widetilde{J}_{\pm \theta}(0)$.  Luckily, we may re-use the calculations from Section \ref{section:residue}, since in \eqref{eq:Jtheta'tilde0def} we used the approximation $\log x \sim \log T$ (with now $T = T_j$).  Hence, the main term of $\sum_{\pm} \widetilde{J}_{\pm \theta}'(0)$ is simply $\log T$ times the main term of $\sum_{\pm} \widetilde{J}_{\pm \theta}(0)$, for which one may then consult \eqref{eq:JMellinResidueCalculation}.

Putting everything together, we obtain the main term as in the right hand side of \eqref{eq:QUEcuspCase}, but multiplied by some absolute constant of proportionality, which we claim is $1$.  One can check this by comparison with \cite[Conjecture 1.1]{Young}, by specializing $q=1$.  Alternatively, one may integrate both sides of \eqref{eq:QUEcuspCase} over $0 \leq x \leq 1$, in which case the result must be consistent with the QUE theorem of Lindenstrauss and Soundararajan.  As discussed in the introduction, the term with $B_q(1, u_j) J_0(\frac{\theta T_j}{y})$ is negligible except on a set of $x$  of measure $0$.  

It may also be worth commenting on the presence of the factor $\delta_j$, which is only implicitly visible in Theorem \ref{thm:QUE}, since the Eisenstein series is even.  This $\delta_j$ factor arises from studying the ``opposite'' diagonal terms (those terms with $m = -n$), which are precisely those terms leading to the secondary main term with $B_q$ and the $J_0$-Bessel function.

We also extend Conjecture \ref{conj:QUEcuspCase} into an analog of Theorem \ref{thm:shiftedQUE} in the obvious way.

\end{document}